\documentclass[12pt]{amsart}

\usepackage{color}
\usepackage{amsfonts}
\usepackage{amssymb}
\usepackage{amsmath}
\usepackage{amsaddr}
\usepackage{hyperref}

\DeclareRobustCommand{\SkipTocEntry}[5]{}

\newcommand{\ds}{\displaystyle}

\newcommand{\cii}[1]{_{ {}_{ #1}}}
\newcommand{\eq}[2][label]{\begin{equation}\label{#1}#2\end{equation}}

\newcommand{\av}[2]{\langle #1\rangle\cii {#2}}
\newcommand{\df}{\buildrel\rm{def}\over=}
\newcommand{\BMO}{{\rm BMO}}
\newcommand{\const}{{\rm const}}
\newcommand{\D}{\mathcal{D}}
\newcommand{\Bel}{\mathbf B}
\newcommand{\bel}{\mathbf b}
\newcommand{\ve}{\varepsilon}
\def\grad{\operatorname{grad}}
\def\Ker{\operatorname{Ker}}

\newcommand{\R}{\mathbb{R}}
\newcommand{\ma}{Monge--Amp\`{e}re }

\newtheorem*{thmnonum}{Theorem}
\newtheorem{theorem}{Theorem}[section]
\newtheorem{cor}[theorem]{Corollary}
\newtheorem{lemma}[theorem]{Lemma}
\numberwithin{equation}{section}
\newtheorem{remark}[theorem]{Remark}

\begin{document}
\thispagestyle{empty}

\begin{titlepage}
{\large \title{Cincinnati lectures on Bellman functions}}

\author[Vasily Vasyunin]{Vasily Vasyunin}
\address{St. Petersburg Department of the V.~A.~Steklov Mathematical Institute, RAS}
\author[]{\tiny Edited by\vspace{-.5cm}}
\author{\footnotesize Leonid Slavin}
\address{\scriptsize University of Cincinnati}
\date{}
\end{titlepage}
\maketitle
\thispagestyle{empty}
\newpage
\addtocontents{toc}{\SkipTocEntry}
\section*{Preface}
In January-March 2011, the Department of Mathematical Science at the University of Cincinnati held a Taft Research Seminar ``Bellman function method in harmonic analysis.''
The seminar was made possible by a generous grant from the Taft Foundation. The principal speaker at the seminar was Vasily Vasyunin. The local host and convener of the seminar was Leonid Slavin.
\medskip

The seminar was in effect a 10-week lecture- and discussion-based course. 
This manuscript represents a slightly revised content of those lectures. In particular, it includes some technical details that were omitted in class due to time constraints.
\newpage
{\large \tableofcontents}
\newpage
\section{Buckley inequality}

The average of a summable positive function (a {\it weight}) $w$ over an interval $I$
will be denoted by the symbol $\av wI$ :
$$
\av wI\df\frac1{|I|}\int_Iw(t)\,dt\,,
$$
where $|I|$ stands for the Lebesgue measure of $I$. For an interval
$J$, the symbol $A_\infty(J,\delta)$ denotes the $\delta$-ball in
the Muckenhoupt class $A_\infty$:
\eq[01]{
A_\infty(J,\delta)\df\left\{w\colon w\in L^1(J),\;w\ge0,\;
\av wI\le\delta e^{\av{\log w}I}\ \forall I\subset J \right\}\,.
}
We denote by $\D_J$ the set of all dyadic subintervals of $J$ and by
$A^d_\infty(J,\delta)$ the dyadic analogue of~\eqref{01}, i.e. in the
definition of $A^d_\infty(J,\delta)$ we consider only $I\in\D_J$.

\begin{thmnonum}[Buckley~\cite{Bu}]
There exists a constant $c=c(\delta)$ such that
$$
\sum_{I\in\D_J}|I|\Big(\frac{\av w{I_+}-\av w{I_-}}{\av wI}\Big)^2\le
c(\delta)|J|
$$
for any weight $w$ from $A^d_\infty(J,\delta)$.
\end{thmnonum}

In the statement of the theorem we use notation $I_\pm$ to mean the
right and left halves of $I$, respectively. By $\D^n_J$ we denote
the $n$-th generation of the dyadic intervals nested in $J$, i.e.
$\D^0_J=\{J\}$, $\D^1_J=\{J_\pm\}$, etc.

Now, we are ready to introduce the main object of our consideration,
the so-called Bellman function of the problem.
\begin{align*}
\Bel(x)&=\Bel(x_1,x_2;\delta)
\\
&\df\!\!\!\!\sup_{w\in A^d_\infty(J,\delta)}
\left\{\frac1{|J|}\sum_{I\in\D_J}|I|
\Big(\frac{\av w{I_+}-\av w{I_-}}{\av wI}\Big)^2\colon
\av wJ=x_1,\;\av{\log w}J=x_2\right\}\,.
\end{align*}
This function is defined on the domain
$$
\Omega_\delta\df\left\{x=(x_1,x_2)\colon
\log\frac{x_1}\delta\le x_2\le\log x_1\right\}\,.
$$
Indeed, the right bound is simply Jensen's inequality and the left
one means that our weight $w$ is from $A_\infty(J,\delta)$. The
parameter $\delta$ is fixed throughout. Let us note that we did
not assign the index $J$ to $\Bel,$ despite the fact that all
test functions $w$ in its definition are considered on $J$. This
omission is not due to our desire to simplify notation, but rather
an indication of the very important fact that the function $\Bel$
does not depend on $J$.

A bit more notation. For a given weight $w\in A_\infty(J,\delta)$
and any subinterval $I\subset J$, there corresponds the following
point of $\Omega_\delta$: $x^I=(\av wI,\av{\log w}I)$.

({\bf Homework assignment:} {\small \sl \color{blue} Check that
the function $\Bel$ defined on the whole domain $\Omega_\delta,$
i.e. for every point $x,$ $x\in\Omega_\delta,$ there exists a
function $w\in A_\infty(J,\delta)$ such that $x=x^J.$})
\medskip

Let us now consider some properties of $\Bel$ that are clear from
its definition; these properties will help us find $\Bel$ explicitly.

\begin{lemma}[Main inequality]
For every pair of points $x^\pm$ from $\Omega_\delta$ such that
their mean $x=(x^++x^-)/2$ is also in $\Omega_\delta$\textup,
the following inequality holds
\eq[main]{
\Bel(x)\ge\frac{\Bel(x^+)+\Bel(x^-)}2+
\Big(\frac{x^+_1-x^-_1}{x_1}\Big)^2.
}
\end{lemma}
\begin{proof}
Let us split the sum in the definition of $\Bel$ into three parts:
the sum over $\D_{J_+}$, the sum over $\D_{J_-},$ and an additional
term, corresponding to $J$ itself:
\begin{align*}
\frac1{|J|}\sum_{I\in\D_J}|I|&
\Big(\frac{\av w{I_+}-\av w{I_-}}{\av wI}\Big)^2
\\
&=\frac1{2|J_+|}\sum_{I\in\D_{J_+}}|I|
\Big(\frac{\av w{I_+}-\av w{I_-}}{\av wI}\Big)^2
\\
&+\frac1{2|J_-|}\sum_{I\in\D_{J_-}}|I|
\Big(\frac{\av w{I_+}-\av w{I_-}}{\av wI}\Big)^2
\\
&+\Big(\frac{\av w{J_+}-\av w{J_-}}{\av wJ}\Big)^2.
\end{align*}
Now we choose the weights $w^\pm$ on the intervals $J_\pm$ that almost
give us the supremum in the definition of $\Bel(x^\pm)$, i.e.
$$
\frac1{|J_\pm|}\sum_{I\in\D_{J_\pm}}|I|
\Big(\frac{\av {w^\pm}{I_+}-\av {w^\pm}{I_-}}{\av {w^\pm}I}\Big)^2\ge
\Bel(x^\pm)-\eta,
$$
for an arbitrary fixed small $\eta>0$. Then for the weight $w$ on
$J,$ defined as $w^+$ on $J_+$ and $w^-$ on $J_-,$ we obtain the
inequality
\eq[02]{
\frac1{|J|}\sum_{I\in\D_J}|I|
\Big(\frac{\av w{I_+}-\av w{I_-}}{\av wI}\Big)^2\ge
\frac{\Bel(x^+)+\Bel(x^-)}2-\eta+\Big(\frac{x^+_1-x^-_1}{x_1}\Big)^2.
}
Observe that the compound weight $w$ is an admissible weight,
corresponding to the point $x$. Indeed, $x^\pm=x^{J_\pm}$ and
by construction $w^\pm\in A^d_\infty(J_\pm,\delta);$ therefore,
the weight $w$ satisfies the inequality $\av wI\le\delta
e^{\av{\log w}I}$ for all $I\in\D_{J_+},$ since $w^+$ does,
and for all $I\in\D_{J_-},$ since $w^-$ does. Lastly,
$\av wJ\le\delta e^{\av{\log w}J}$, because, by assumption,
$x\in\Omega_\delta$.

We can now take supremum in~\eqref{02} over all admissible
weights $w,$ which yields
$$
\Bel(x)\ge\frac{\Bel(x^+)+\Bel(x^-)}2-\ve+
\Big(\frac{x^+_1-x^-_1}{x_1}\Big)^2,
$$
which proves the main inequality because $\eta$ is arbitrarily small.
\end{proof}

\begin{lemma}[Boundary condition]
\label{bcBI}
$$
\Bel(x_1,\log x_1)=0\,.
$$
\end{lemma}

\begin{proof}
Let us take a boundary point $x$ of our domain $\Omega_\delta,$ that is a point with
$x_2=\log x_1.$ Since the equality in Jensen's inequality $e^{\av{w}{}}
\le\av{e^w}{}$ occurs only for constant functions $w,$ the only test function corresponding to $x$ is the 
constant weight $w=x_1.$ So,
on this boundary we have $\Bel(x)=0.$
\end{proof}

\begin{lemma}[Homogeneity]
There is a function $g$ on $[1,\delta]$ satisfying $g(1)=0$ and
such that
$$
\Bel(x)=\Bel(x_1e^{-x_2},0)=g(x_1e^{-x_2})\,.
$$
\end{lemma}
\begin{proof}
For a weight $w$ on an interval $J$ and a positive number $\tau$
consider a new weight, $\tilde w=\tau w$. If $x$ is a point from
$\Omega_\delta$ corresponding to $w$ and $J,$ i.e. $x_1=\av wJ,$
$x_2=\av{\log w}J,$ then the point $\tilde x=(\tilde x_1,\tilde x_2),$
$\tilde x_1 =\tau x_1,$ $\tilde x_2=x_2+\log\tau,$
corresponds to $\tilde w.$ Note that the expression in the definition
of $\Bel$ is homogeneous of order 0 with respect to $w,$ i.e. it does
not depend on $\tau.$ Since the weights $w$ and $\tilde w$ run over
the whole set $A^d_\infty(J,\delta)$ simultaneously, we get
$\Bel(x)=\Bel(\tilde x).$ Choosing $\tau=e^{-x_2},$ we obtain
$$
\Bel(x)=\Bel(x_1e^{-x_2},0)\,.
$$
To complete the proof, it suffices to take $g(s)=\Bel(s,0).$
The boundary condition $g(1)=0$ holds due to Lemma~\ref{bcBI}.
\end{proof}

We are now ready to demonstrate how the Bellman function method works.

\begin{lemma}[Bellman induction]
\label{Bellman induction BI}
Let $g$ be a nonnegative function on $[1,\delta]$ such that the
function $B(x)\df g(x_1e^{x_2})$ satisfies inequality~\eqref{main}
in $\Omega_\delta.$ Then Buckley's inequality holds with the
constant $c(\delta)=\|g\|_{L^\infty([1,\delta])}.$
\end{lemma}
\begin{proof}
Fix an interval $J$ and a weight $w\in A^d_\infty(J,\delta).$ Let
us repeatedly use the main inequality in the form
$$
|I|\,B(x^I)\ge |I_+|\,B(x^{I_+})+|I_-|\,B(x^{I_-})+|I|
\Big(\frac{x^{I_+}_1-x^{I_-}_1}{x^I_1}\Big)^2,
$$
applying it first to $J$, then to the intervals of the first
generation (that is $J_\pm$), and so on until $\D^n_J:$
\begin{align*}
|J|\,B(x^J)&\ge |J_+|\,B(x^{J_+})+|J_-|\,B(x^{J_-})+|J|
\Big(\frac{x^{J_+}_1-x^{J_-}_1}{x^J_1}\Big)^2
\\
&\ge\sum_{I\in\D^n_J}|I|\,B(x^I)+\sum_{k=0}^{n-1}\sum_{I\in\D^k_J}
|I|\Big(\frac{x^{I_+}_1-x^{I_-}_1}{x^I_1}\Big)^2.
\end{align*}
Therefore,
$$
\sum_{k=0}^{n-1}\sum_{I\in\D^k_J}
|I|\Big(\frac{x^{I_+}_1-x^{I_-}_1}{x^I_1}\Big)^2\le |J|\,B(x^J)\,,
$$
and, passing to the limit as $n\to\infty$, we get
$$
\sum_{I\in\D_J}
|I|\Big(\frac{x^{I_+}_1-x^{I_-}_1}{x^I_1}\Big)^2\le |J|\,B(x^J)=
|J|\,g(x_1e^{-x_2})\le|J|\sup_{s\in[1,\delta]}g(s)\,.
$$
\end{proof}
A natural question arises: how to find such a function $g$?
To answer it, we first replace our main inequality, which is
an inequality in finite differences, by a differential inequality.
Let us denote the difference between $x^+$ and $x^-$ by $2\Delta,$
then $x^\pm=x\pm\Delta$ and the Taylor expansion around the point
$x$ gives
$$
B(x^\pm)=B(x)\pm\frac{\partial B}{\partial x_1}\Delta_1
\pm\frac{\partial B}{\partial x_2}\Delta_2
+\frac12\frac{\partial^2 B}{\partial x_1^2}\Delta_1^2
+\frac{\partial^2 B}{\partial x_1\partial x_2}\Delta_1\Delta_2
+\frac12\frac{\partial^2 B}{\partial x_2^2}\Delta_2^2+o(|\Delta|^2)\,,
$$
and, therefore,
\begin{align*}
&\frac{B(x^+)+B(x^-)}2+\Big(\frac{x^+_1-x^-_1}{x_1}\Big)^2-B(x)
\\
&\qquad=\frac12\frac{\partial^2 B}{\partial x_1^2}\Delta_1^2
+\frac{\partial^2 B}{\partial x_1\partial x_2}\Delta_1\Delta_2
+\frac12\frac{\partial^2 B}{\partial x_2^2}\Delta_2^2+
4\Big(\frac{\Delta_1}{x_1}\Big)^2
+o(|\Delta|^2)\,.
\end{align*}
Thus, under the assumption that our candidate $B$ is sufficiently
smooth, the main inequality~\eqref{main} implies the following
matrix differential inequality
\eq[matr]{
\left(
\begin{matrix}
\ds\frac{\partial^2 B}{\partial x_1^2}+\frac8{\;x_1^2}&\ &
\ds\frac{\partial^2 B}{\;\partial x_1\partial x_2}
\\
&&
\\
\ds\frac{\partial^2 B}{\;\partial x_1\partial x_2}&\ &
\ds\frac{\partial^2 B}{\partial x_2^2}
\end{matrix}
\right)\le0\,.
}

By the preceding two lemmata, we can restrict our search to
functions $B$ of the form $B(x_1,x_2)=g(x_1e^{-x_2}),$ where $g$
is a function on the interval $[1,\delta].$ In terms of $g,$ our
condition~\eqref{matr} can be rewritten as follows:
$$
\left(
\begin{matrix}
\ds e^{-2x_2}\Big(g''+\frac8{s^2}\Big)&\ &
\ds-e^{-x_2}(sg')'
\\
&&
\\
\ds-e^{-x_2}(sg')' &\ &
\ds s(sg')'
\end{matrix}
\right)\le0\,,
$$
where $g=g(s)$ and $s=x_1e^{-x_2}$. This matrix inequality is
equivalent to three scalar inequalities:
\begin{align}
\label{B11}
g''+\frac8{s^2}&\le0,
\\
\label{B22}
(sg')'&\le0,
\end{align}
and the condition that the determinant of the matrix must be
nonnegative. However, we replace the last requirement by a
stronger one --- we require the determinant to be identically
zero. This requirement comes from our desire to find the best
possible estimate: if we take an extremal weight $w,$ i.e. a
weight on which the supremum in the definition of the Bellman
function is attained, then we must have equalities on each step
of the Bellman induction; therefore, on each step the main
inequality~\eqref{main} becomes equality. Thus, for each dyadic
subinterval $I$ of $J$ there exists a direction through the point
$x^I$ in $\Omega_\delta$ along which the quadratic form given
by~\eqref{matr} is identically zero. Hence, the matrix~\eqref{matr}
has a non-trivial kernel and so must have a zero determinant.

Calculating the determinant, we get the equation
$$
\Big(g'-\frac8s\Big)(sg')'=0\,.
$$
The general solution of this equation is $g(s)=c\log s+ c_1.$ Due to the boundary condition $g(1)=0,$
we have to take $c_1=0.$ 

Now we need to chose another constant, $c.$ To this end, we return
to the necessary conditions~\eqref{B11}--\eqref{B22}. The second
inequality is fulfilled for all $c,$ because the expression is
identically zero, while the first one gives $c\ge8.$ Since we would like
to have $g$ as small as possible (as it gives the upper bound in
Buckley's inequality), it is natural to take $c=8.$ Finally, we get
$$
g(s)=8\log s\qquad\text{and}\qquad B(x_1,x_2)=8(\log x_1-x_2)\,.
$$

\begin{lemma}
The function
$$
B(x_1,x_2)=8(\log x_1-x_2)
$$
satisfies the main inequality~\eqref{main}.
\end{lemma}
\begin{proof}
Put, as before, $\Delta=\frac12(\,x^+-x^-)$, so $x^\pm=x\pm\Delta$. Then
\begin{align*}
B(x)&-\frac{B(\,x^+)+B(\,x^-)}2-\Big(\frac{x^+_1-x^-_1}{x_1}\Big)^2
\\
&=8\log x_1-8x_2-4\log(\,x^+_1x^-_1)+4(\,x^+_2+x^-_2)
-\Big(\frac{x^+_1-x^-_1}{x_1}\Big)^2
\\
&=4\log\frac{x_1^2}{(x_1+\Delta_1)(x_1-\Delta_1)}
-4\Big(\frac{\Delta_1}{x_1}\Big)^2
\\
&=-4\left[\log\Big(1-\Big(\frac{\Delta_1}{x_1}\Big)^2\Big)+
\Big(\frac{\Delta_1}{x_1}\Big)^2\right]\ge0\,.\rule{0pt}{25pt}
\end{align*}
\end{proof}

Now we can apply Lemma~\ref{Bellman induction BI} to $g(s)=8\log s,$
which yields the following
\begin{thmnonum}
The estimate
$$
\sum_{I\in\D_J}|I|\Big(\frac{\av w{I_+}-\av w{I_-}}{\av wI}\Big)^2\le
8\log\delta\,|J|
$$
holds for any weight $w\in A^d_\infty(J,\delta)$.
\end{thmnonum}

Concluding this section, I would like to emphasize that we still
have not found the Bellman function $\Bel.$ The theorem just proved
guarantees only the estimate $$\Bel(x)\le8(\log x_1-x_2).$$

\section{Homework assignment: A simple two-weight inequality}
{{\small\sl\color{blue} As an exercise, verify every step, outlined
below, of the proof of this theorem:

\begin{thmnonum}
If two weights $u,v\in L^1(J)$ satisfy the condition
$$
\sup_{I\in\D_J}\av uI\av vI\le M^2\,,
$$
then
$$
\frac1{|J|}\sum_{I\in\D_J}|I|\,|\av u{I_+}-\av u{I_-}|\,
|\av v{I_+}-\av v{I_-}|\le16M\sqrt{\av uJ\av vJ}\,.
$$
\end{thmnonum}

\subsection{Remark on the Haar functions}
If we introduce the normalized Haar system
$$
h\cii I(t)=\frac1{\sqrt{|I|}}
\begin{cases}
-1\quad&\text{if }\ t\in I_-,
\\
\phantom{-}1\quad&\text{if }\ t\in I_+,
\end{cases}
$$
then $\sqrt{|I|}(\,\av w{I_+}-\av w{I_-})=2(\,w,h\cii I)$.
Thus the statement of the Theorem above can be rewritten in the form
$$
\frac1{|J|}\sum_{I\in\D_J}|(u,h\cii I)|\,|(v,h\cii I)|
\le4M\sqrt{\av uJ\av vJ}
$$
and that of Buckley's inequality, in the form
$$
\sum_{I\in\D_J}\left(\frac{(w,h\cii I)}{\;\av wI}\right)^2
\le2\log\delta\,|J|\,.
$$

\subsection{The Bellman function of the problem}
$$
\Bel(x;m,M)\df\sup_{u,v}\Big\{\frac1{|J|}
\sum_{I\in\D_J}|(u,h\cii I)|\,|(v,h\cii I)|\Big\}\,,
$$
where the supremum is taken over the set of all admissible pairs
of weights, i.e. such pairs $u,v$ that $\av uJ=x_1,$ $\av vJ=x_2,$
and $m^2\le\av uI\av vI\le M^2,$ $\forall I\in\D_J.$ To prove the
theorem means to prove the inequality
$$
\Bel(x;0,M)\le4M\sqrt{x_1x_2}\,.
$$
The domain of $\Bel$ is
$$
\Omega=\big\{x=(x_1,x_2)\colon m^2\le x_1x_2\le M^2\big\}\,.
$$

\subsection{Properties}
\begin{itemize}
\item The function $\Bel$ does not depend on $J$.
\item Homogeneity: $\Bel(x_1,x_2)=\Bel(x_1x_2,1)\df g(x_1x_2)$.
\item Boundary condition: $\Bel|_{x_1x_2=m^2}=g(m^2)=0$.\rule{0pt}{13pt}
\end{itemize}
\smallskip

\subsection{Main inequality}~

\noindent For every pair $x^\pm\in\Omega$ such that
$x=\frac12x^++\frac12x^-\in\Omega$, we have
$$
\Bel(x)\ge\frac{\Bel(x^+)+\Bel(x^-)}2+
\frac{|x^+_1-x^-_1|\,|x^+_2-x^-_2|}4.
$$
In the differential form,
$$
\left(
\begin{matrix}
\ds\frac{\partial^2 B}{\partial x_1^2}&\ &
\ds\frac{\partial^2 B}{\;\partial x_1\partial x_2}\pm1
\\
&&
\\
\ds\frac{\partial^2 B}{\;\partial x_1\partial x_2}\pm1&\ &
\ds\frac{\partial^2 B}{\partial x_2^2}
\end{matrix}
\right)\le0\,,
$$
or, in terms of $g,$
$$
\left(
\begin{matrix}
\ds x_2^2g''&\ &
\ds g'+x_1x_2g''+\sigma
\\
&&
\\
\ds g'+x_1x_2g''+\sigma&\ &
\ds x_1^2g''
\end{matrix}
\right)\le0\,,
$$
where $\sigma=\pm1$.

The condition that this matrix be degenerate gives us a differential
equation, whose general solution is $g(s)=2c\sqrt{s}-s+c_1$ (this
is quite a bit of work).  The constant $c_1$ can be found from the
boundary condition: $c_1=m^2-2cm,$ and the constant $c$ has to be
chosen as small as possible to obtain the best estimate: $c=2M.$
Thus, the answer is
$$
\Bel(x;m,M)\le4M\sqrt{x_1x_2}-x_1x_2+m^2-4mM\,.
$$
}}

All details of the proof that, in fact, we have found the true Bellman function, i.e.
$$
\Bel(x;m,M)=4M\sqrt{x_1x_2}-x_1x_2+m^2-4mM\,,
$$
can be found in~\cite{VaVo1}.

\section{John--Nirenberg inequality, Part I}

A function $\varphi\in L^1(J)$ is said to belong to the space $\BMO(J)$ if
$$
\sup_I\av{|\varphi(s)-\av\varphi{I}|}I<\infty
$$
for all subintervals $I\subset J.$ If this condition holds only for the dyadic
subintervals $I\in\D_J,$ we will write $\varphi\in\BMO^d(J)$.
In fact, the following is true
$$
\varphi\in \BMO(J)\iff \Big(\int_I|\varphi(s)-\av{\varphi}I|^p\,
ds\Big)^{\frac1p}<\infty,\quad\forall p\in(0,\infty),\ I\subset J\,.
$$
If we factor over the constants, we get a normed space, where the
expression on the right-hand side can be taken as one of the
equivalent norms for any $p\in[1,\infty).$
In what follows, we will use the $L^2$-based norm:
$$
\|\varphi\|_{\BMO(J)}^2=
\sup_{I\subset J}\frac1{|I|}\int_I|\varphi(s)-\av{\varphi}I|^2\,ds=
\sup_{I\subset J}\left(\av{\varphi^2}I-\av{\varphi}I^2\right)\,.
$$
The $\BMO$ ball of radius $\ve$ centered at $0$ will be denoted by
$\BMO_\ve.$ Using the Haar decomposition
$$
\varphi(s)=\av{\varphi}J+\sum_{I\in\D_J}(\varphi,h\cii I)h\cii I(s)\,,
$$
we can write down the expression for the norm in the following way
$$
\|\varphi\|_{\BMO(J)}^2=
\sup_{I\subset J}\frac1{|I|}\sum_{L\in\D_I}|(\varphi,h\cii L)|^2 =\frac14
\sup_{I\subset J}\frac1{|I|}\sum_{L\in\D_I}|L|\,
\big(\,\av\varphi{L^+}-\av\varphi{L^-}\big)^2.
$$

\begin{thmnonum}[John--Nirenberg \cite{JoNi}]
There exist absolute constants $c_1$ and $c_2$ such that
$$
\left|\left\{s\in J\colon|\varphi(s)-\av{\varphi}J|
\ge\lambda\right\}\right|\le c_1e^{-c_2\frac\lambda{\|\varphi\|}}|J|
$$
 for all $\varphi\in\BMO_\ve(J).$
\end{thmnonum}
An equivalent, integral form of the same assertion is the following
\begin{thmnonum}
There exists an absolute constant $\ve_0$ such that for any
$\varphi\in\BMO_\ve(J)$ with $\ve<\ve_0$ the inequality
$$
\av{e^\varphi}J\le c\,e^{\av{\varphi}J}
$$
holds with a constant $c=c(\ve)$ not depending on $\varphi$.
\end{thmnonum}

We shall prove the theorem in this integral form and find the
sharp constant $c(\ve).$ Our Bellman function,
$$
\Bel(x;\ve)\;\;\df\sup_{\varphi\in\BMO_\ve(J)}\left\{\av{e^\varphi}J
\colon\av{\varphi}J=x_1,\;\av{\varphi^2}J=x_2\right\},
$$
is well-defined on the domain
$$
\Omega_\ve\df\left\{x=(x_1,x_2)\colon x_1^2\le x_2\le x_1^2+\ve^2\right\}\,.
$$

First, we will consider the dyadic problem and deduce the main
inequality for the dyadic Bellman function.

\begin{lemma}[Main inequality]
For every pair of points $x^\pm$ from $\Omega_\ve$ such that their
mean $x=(x^++x^-)/2$ is also in $\Omega_\ve$\textup, the following
inequality holds
\eq[mainJN]{
\Bel(x)\ge\frac{\Bel(x^+)+\Bel(x^-)}2\,.
}
\end{lemma}
\begin{proof}
The proof repeats almost verbatim the proof of the main inequality for the
Buckley Bellman function. We split the integral in the definition of
$\Bel$ into two parts, the integral over $J_+$ and the one over $J_-:$
$$
\int_J e^{\varphi(s)}\,ds=\int_{J_+}\!\!e^{\varphi(s)}\,ds+
\int_{J_-}\!\!e^{\varphi(s)}\,ds\,.
$$
Now we choose such functions $\varphi^\pm$ on the intervals $J_\pm$
that they almost give us the supremum in the definition of $\Bel(x^\pm),$ i.e.
$$
\frac1{|J_\pm|}\int_{J_\pm}\!\!e^{\varphi(s)}\,ds\ge\Bel(x^\pm)-\eta,
$$
for a fixed small $\eta>0$. Then for the function $\varphi$ on $J,$
defined as $\varphi^+$ on $J_+$ and $\varphi^-$ on $J_-,$ we obtain
the inequality
\eq[02JN]{
\frac1{|J|}\int_J e^{\varphi(s)}\,ds\ge
\frac{\Bel(x^+)+\Bel(x^-)}2-\eta\,.
}
Observe that the compound function $\varphi$ is an admissible test
function corresponding to the point $x.$ Indeed, $x^\pm=x^{J_\pm}$
and by construction $\varphi^\pm\in\BMO^d_\ve(J_\pm);$ therefore,
the function $\varphi$ satisfies the inequality $\av{\varphi^2}I-
\av{\varphi}I^2\le \ve^2$ for all $I\in\D_{J_+},$ since $\varphi^+$
does, and for all $I\in\D_{J_-},$ since $\varphi^-$ does. Lastly,
$\av{\varphi^2}J-\av{\varphi}J^2\le\ve^2$, because, by assumption,
$x\in\Omega_\ve$.

We can now take supremum in \eqref{02JN} over all admissible functions
$\varphi$ which yields
$$
\Bel(x)\ge\frac{\Bel(x^+)+\Bel(x^-)}2-\eta\,,
$$
which proves the main inequality because $\eta$ is arbitrarily small.
\end{proof}

As in the case of the Buckley inequality, the next our step is to
derive a boundary condition for $\Bel.$

\begin{lemma}[Boundary condition]
\eq[bcJN]{
\Bel(x_1,x_1^2)=e^{x_1}\,.
}
\end{lemma}
\begin{proof}
The function $\varphi(s)=x_1$ is the only test function corresponding
to the point $x=(x_1,x_1^2)$, because the equality in the H\"older
inequality $x_2\ge x_1^2$ occurs only for constant functions.
Hence, $e^\varphi=e^{x_1}.$
\end{proof}

Now we are ready to describe super-solutions as functions verifying
the main inequality and the boundary conditions.

\begin{lemma}[Bellman induction]
If $B$ is a continuous function on the domain $\Omega_\ve,$
satisfying the main inequality~\eqref{mainJN} for any pair
$x^\pm$ of points from $\Omega_\ve$ such that $x\df\frac{x^++x^-}2\in\Omega_\ve,$
as well as the boundary condition~\eqref{bcJN}\textup, then $\Bel(x)\le B(x).$
\end{lemma}

\begin{proof}
Fix a bounded function $\varphi\in\BMO_\ve(J)$. By the main
inequality we have
$$
|J|B(x^J)|\ge|J_+|B(x^{J_+})|+|J_-|B(x^{J_-})|
\ge\sum_{I\in\D_{\!J}^n}|I|B(x^I)=\int_JB(x^{(n)}(s))\,ds\,,
$$
where $x^{(n)}(s)=x^I,$ when $s\in I,$ $I\in\D_{\!J}^n.$ (Recall that
$\D_{\!J}^n$ stands for the set of subintervals of $n$-th generation.)
By the Lebesgue differentiation theorem we have
$x^{(n)}(s)\to(\varphi(s),\varphi^2(s))$ almost everywhere. Now,
we can pass to the limit in this inequality as $n\to\infty$. Since
$\varphi$ is assumed to be bounded, $x^{(n)}(s)$ runs in a bounded
--- and, therefore, compact --- subdomain of $\Omega_\ve.$ Since $B$
is continuous, it is bounded on any compact set and so, by the
Lebesgue dominated convergence theorem, we can pass to the limit
in the integral using the boundary condition~\eqref{bcJN}:
\eq[upper]{
|J|B(x^J)\ge\int_JB(\varphi(s),\varphi^2(s))\,ds=
\int_Je^{\varphi(s)}ds=|J|\av{e^\varphi}J\,.
}

To complete the proof of the lemma, we need to pass from bounded
to arbitrary \BMO~test functions. To this end, we will use the
following result:
\begin{lemma}[Cut-off Lemma]
\label{norm}
Fix $\varphi\in\BMO(J)$ and two real numbers $c,d$ such that $c<d.$
Let $\varphi_{c,d}$ be the cut-off of $\varphi$ at heights $c$ and $d:$
\eq[cutoff]{
\varphi_{c,d}(s)=
\begin{cases}
\ c,&if~\varphi(s)\le c;\\
\varphi(s),&if~c<\varphi(s)<d;\\
\ d,&if~\varphi(s)\ge d.
\end{cases}
}
Then
$$
\av{\varphi_{c,d}^2}I-\av{\varphi_{c,d}}I^2\le
\av{\varphi^2}I-\av\varphi I^2,\quad\forall I,\ I\subset J,
$$
and\textup, consequently\textup,
$$
\|\varphi_{c,d}\|_{\BMO}\le\|\varphi\|_{\BMO}.
$$
\end{lemma}

\begin{proof}
First, let us note that it is sufficient to prove this lemma for
a one-sided cut, for example, for $c=-\infty.$ We then get the full
statement by applying this argument twice. Indeed, if we denote by
$C_d\varphi$ the cut-off of $\varphi$ from above at height $d,$
i.e. $C_d\varphi=\varphi_{-\infty,d},$ then
$\varphi_{c,d}=-C_{-c}(-C_d\varphi$).

Take a measurable subset $I\subset J$ and let $I_1=\{s\in I\colon
\varphi(s)<d\}$ and $I_2=\{s\in I\colon \varphi(s)\ge d\}.$ Let
$\beta_k=|I_k|/|I|, k=1,2.$
We have the following identity:
\begin{align*}
\bigl[\av{\varphi^2}I&-\av\varphi I^2\bigr]
-\bigl[\av{(C_d\varphi)^2}I-\av{C_d\varphi}I^2\bigr]
\\
=&\beta_2\bigl[\av{\varphi^2}{I_2}-\av\varphi{I_2}^2\bigr]+
\beta_1\beta_2\bigl[\av\varphi{I_2}-d\bigr]\bigl[\av\varphi{I_2}+d-
2\av\varphi{I_1}\bigr],
\end{align*}
which proves the lemma, because $\av\varphi{I_1}\!\!\le d\le
\av\varphi{I_2}.$
\end{proof}

Now, let $\varphi\in\BMO_\ve(J)$ be a function bounded from above.
Then, by the above lemma, $\varphi_n\df\varphi_{-n,\infty}\in\BMO_\ve(J),$
and, according to~\eqref{upper}, we have
$$
B(\av{\varphi_n}J,\av{\varphi_n^2}J)\ge\av{e^{\varphi_n}}J\,.
$$
Since $e^\varphi$ is a summable majorant for $e^{\varphi_n}$ and $B$
is continuous, we can pass to the limit and obtain the
estimate~\eqref{upper} for any function $\varphi$ bounded from above.
Finally, we repeat this approximation procedure for an arbitrary
$\varphi$. Now, we take $\varphi_n=\varphi_{-\infty,n}$ and use the
monotone convergence theorem to pass to the limit in the right-hand
side of the inequality.

So, we have proved the inequality
$$
B(x^J)\ge\av{e^\varphi}J
$$
for arbitrary $\varphi\in\BMO_\ve(J).$ Taking supremum over all admissible
test functions corresponding to the point $x,$ we get $B(x)\ge\Bel(x).$
\end{proof}

As before, we pass from the finite-difference inequality~\eqref{mainJN}
to the infinitesimal one:
\eq[matrJN]{
\frac{d^2B}{dx^2}\df\left(
\begin{matrix}
\ds\frac{\partial^2 B}{\partial x_1^2}&\ &
\ds\frac{\partial^2 B}{\;\partial x_1\partial x_2}
\\
&&
\\
\ds\frac{\partial^2 B}{\;\partial x_1\partial x_2}&\ &
\ds\frac{\partial^2 B}{\partial x_2^2}
\end{matrix}
\right)\le0\,,
}
and we will require this Hessian matrix to be degenerate, i.e.
$\det(\frac{d^2B}{dx^2})=0.$ Again, to solve this PDE, we use a
homogeneity property to reduce the problem to an ODE.

\begin{lemma}[Homogeneity]
There exists a function $G$ on the interval $[0,\ve^2]$ such that
$$
\Bel(x;\ve)=e^{x_1}G(x_2-x_1^2)\,,\qquad G(0)=1\,.
$$
\end{lemma}
\begin{proof}
Let $\varphi$ be an arbitrary test function and $x=(\av{\varphi}J,\av{\varphi^2}J)$ its
Bellman point on $J.$ Then the function $\tilde\varphi\df\varphi+\tau$ is also a test function
with the same norm, and its Bellman point is $\tilde x=(x_1+\tau, x_2+2\tau x_1+\tau^2)$.
Therefore,
$$
\Bel(\tilde x)=\sup_{\tilde\varphi}\av{e^{\tilde\varphi}}J=
e^\tau\sup_{\varphi}\av{e^\varphi}J=e^\tau\Bel(x)\,.
$$
Choosing $\tau=-x_1$ we get
$$
\Bel(x)=e^{-\tau}\Bel(x_1+\tau,x_2+2\tau x_1+\tau^2)=
e^{x_1}\Bel(0,x_2-x_1^2)\,.
$$
Setting $G(t)=\Bel(0,t)$ completes the proof.
\end{proof}
Since $G>0$, we can introduce $g(t)=\log G(t)$ and look for a
function $B$ of the form
$$
B(x_1,x_2)=e^{x_1+g(x_2-x_1^2)}\,.
$$
By direct calculation, we get
\begin{align*}
\frac{\partial^2 B}{\partial x_1^2}=&
\left(1-4x_1g'+4x_1^2(g')^2-2g'+4x_1^2g''\right)B\,,
\\
\frac{\partial^2 B}{\;\partial x_1\partial x_2}=&
\left(g'-2x_1(g')^2-2x_1g''\right)B\,,
\\
\frac{\partial^2 B}{\partial x_2^2}=&\left((g')^2+g''\right)B.
\end{align*}
The partial differential equation $\det(\frac{d^2B}{dx^2})=0$ then
turns into the following ordinary differential equation:
$$
\left(1-4x_1g'+4x_1^2(g')^2-2g'+4x_1^2g''\right)\left((g')^2+g''\right)
=\left(g'-2x_1(g')^2-2x_1g''\right)^2,
$$
which reduces to
$$
g''-2g'g''-2(g')^3=0\,.
$$
Dividing by $2(g')^3$ (since we are not interested in constant
solutions), we get
$$
\left(\frac1{g'}-\frac1{4(g')^2}\right)'=1\,,
$$
which yields
$$
\frac1{g'}-\frac1{4(g')^2}=t+\const
$$
or, equivalently,
$$
-\left(1-\frac1{2g'}\right)^2=t+\const,\qquad\forall s\in[0,\ve^2]\,.
$$
Since the left-hand side is non-positive, the constant cannot be
greater than $-\ve^2$. Let us denote it by $-\delta^2,$ where
$\delta\ge\ve.$

Thus, we have two possible solutions:
$$
1-\frac1{2g'_\pm}=\pm\sqrt{\delta^2-t}\,.
$$
Using the boundary condition $g(0)=0$, we obtain
$$
g_\pm(t)=\frac12\int_0^t\frac{ds}{1\mp\sqrt{\delta^2-s}}=
\log\frac{1\mp\sqrt{\delta^2-t}}{1\mp\delta} \pm\sqrt{\delta^2-t}\mp\delta\,.
$$
This yields two solutions for $B:$
$$
B_\pm(x)=\frac{1\mp\sqrt{\delta^2-x_2+x_1^2}}{1\mp\delta}
\exp\left\{x_1\pm\sqrt{\delta^2-x_2+x_1^2}\mp\delta\right\}\,.
$$
\bigskip
{\bf Homework assignment.}

{\small\sl
\begin{enumerate}
\item
{\color{blue}
Check that the quadratic form of the Hessian is:
$$
\sum_{i,j=1}^2\!\!\frac{\partial^2 B_\pm}{\;\partial x_i\partial x_j}
\Delta_i\Delta_j
=\mp\frac{\left(\!\big(x_1\!\pm\!\sqrt{\delta^2\!-\!x_2\!+\!x_1^2}\big)
\Delta_1\!\!-\!\frac12\Delta_2\right)^2\!\!} {\sqrt{\delta^2-x_2+x_1^2}
(1\mp\delta)}
\exp\!\left\{x_1\!\pm\!\sqrt{\delta^2\!-\!x_2\!+\!x_1^2}\mp\delta\right\}.
$$
}
\item
{\color{blue}
Find the extremal trajectories along which the Hessian  degenerates.
}
\end{enumerate}
}

\section{Homogeneous Monge--Amp\`ere equation}

Now, we change the subject of our consideration for a while and look
for the solutions of the equation
\eq[MA]{
B_{x_1x_1}B_{x_2x_2}=(B_{x_1x_2})^2
}
in a general setting.

Linear functions always satisfy~\eqref{MA}. Since we are looking for
the smallest possible concave function $B$, it always will be linear, if
a linear function satisfies the required boundary conditions. It is
a simple case, and in what follows we assume that $B$ is not linear.
This means that in each point $x$ of the domain there exists a unique
(up to a scalar coefficient) vector, say $\Theta(x)$, from the kernel
of the matrix $\frac{d^2\!B}{dx^2}$.

Let us check that functions $B_{x_i}$ are constant along the vector
field $\Theta$. The tangent vector field to the level set
$f(x_1,x_2)=\const$ has the form
$\left(\begin{matrix}
-f_{x_2}\\f_{x_1}
\end{matrix}\right)$
(it is orthogonal to $\grad f=
\left(\begin{matrix}
f_{x_1}\\f_{x_2}
\end{matrix}\right)$).
Thus, we need to check that the both vectors
$\left(\begin{matrix}
-(B_{x_i})_{x_2}\\(B_{x_i})_{x_1}
\end{matrix}\right)$ are in the kernel of the Hessian (i.e.
proportional the kernel vector $\Theta$). This is a direct consequence
of~\eqref{MA}. For example, for $i=1$ we have:

$$
\left(\begin{matrix}
B_{x_1,x_1}&B_{x_1x_2}\\B_{x_2x_1}&B_{x_2x_2}
\end{matrix}\right)
\left(\begin{matrix}
-(B_{x_1})_{x_2}\\(B_{x_1})_{x_1}
\end{matrix}\right)
=\left(\begin{matrix}
-B_{x_1x_1}B_{x_1x_2}+B_{x_1x_2}B_{x_1x_1}\\
-B_{x_2x_1}B_{x_1x_2}+B_{x_2x_2}B_{x_1x_1}
\end{matrix}\right)=0.
$$
\smallskip

If we parameterize the integral curves of the field $\Theta$ by some
parameter $s$ we can write $B_{x_i}\df t_i(s)$, $s=s(x_1,x_2)$. Any
$B_{x_i}$ that is not identically constant can itself be taken as
$s$. However, usually it is more convenient to parameterize the integral
curves by some other parameter with a clear geometrical meaning.

Now, we check that the function $t_0\df B-x_1t_1-x_2t_2$ is also
constant along the integral curves. Since
$$
-\frac{\partial t_0}{\partial x_2}=
-B_{x_2}+x_1\frac{\partial t_1}{\partial x_2}+
x_2\frac{\partial t_2}{\partial x_2}+t_2=
x_1B_{x_1x_2}+x_2B_{x_2x_2}
$$
and
$$
\frac{\partial t_0}{\partial x_1}=
B_{x_1}-t_1-x_1\frac{\partial t_1}{\partial x_1}-
x_2\frac{\partial t_2}{\partial x_1}=
-x_1B_{x_1x_1}-x_2B_{x_1x_2},
$$
we have
$$
\left(\begin{matrix}
-(t_0)_{x_2}\\\;\;(t_0)_{x_1}
\end{matrix}\right)=-x_1\left(\begin{matrix}
-(t_1)_{x_2}\\\;\;(t_1)_{x_1}
\end{matrix}\right)-x_2\left(\begin{matrix}
-(t_2)_{x_2}\\\;\;(t_2)_{x_1}
\end{matrix}\right)\in\Ker\frac{d^2\!B}{dx^2}.
$$

So, we have proved that in the representation
\eq[solMA]{
B=t_0+x_1t_1+x_2t_2
}
of a solution of the homogeneous Monge--Amp\`ere equation, the
coefficients $t_i$ are constant along the vector field generated
by the kernel of the Hessian. Now we prove that the integral curves
of this vector field are in fact straight lines given by the equation
\eq[extrem]{
dt_0+x_1dt_1+x_2dt_2=0\,.
}
This is, indeed, the equation of a straight line, because all the
differentials are constant along the trajectory. In a parametrization of the trajectories is chosen, this equation can
be rewritten as a usual linear equation with constant coefficients. For example,
let us take $s=t_0$; then~\eqref{extrem} turns into
$$
1+x_1\frac{dt_1}{dt_0}+x_2\frac{dt_2}{dt_0}=0\,,
$$
where the coefficients $\frac{dt_i}{dt_0}$, being functions of $t_0$,
are constant on each trajectory.

Now, let us deduce equation~\eqref{extrem}. On one hand,
\eq[d1]{
dB=B_{x_1}dx_1+B_{x_2}dx_2=t_1dx_1+t_2dx_2\,.
}
On the other hand, from representation~\eqref{solMA} we have
\eq[d2]{
dB=dt_0+t_1dx_1+x_1dt_1+t_2dx_2+x_2dt_2\,.
}
A comparison of~\eqref{d1} and~\eqref{d2} yields~\eqref{extrem}.

\bigskip

More details about solutions of the homogeneous Monge--Amp\`ere equation, together
with an example of its application to the John--Nirenberg inequality, can be found
in~\cite{VaVo2}. In the following section, we just consider this alternative
method of finding a candidate for the role of the Bellman function for the
integral John--Nirenberg inequality.

\section{John--Nirenberg inequality, Part II}
Let us now re-solve the Monge--Amp\`ere boundary value problem for
the John--Nirenberg inequality using the method described in the
previous section. We are looking for a solution of the form
$B(x)=t_0+x_1t_1+x_2t_2$ satisfying the boundary condition
\eq[bc]{
B(x_1,x_1^2)=e^{x_1}
}
and the homogeneity condition
$$
B(x_1+\tau,x_2+2\tau x_1+\tau^2)=e^\tau B(x)\,.
$$
This time, instead of using this identity to reduce the number of
variables, we differentiate it with respect to $\tau,$
$$
\frac{\partial B}{\partial x_1}+
(2x_1+2\tau)\frac{\partial B}{\partial x_2}=e^\tau B(x)\,,
$$
and set $\tau=0:$
$$
t_1+2x_1t_2=t_0+x_1t_1+x_2t_2\,.
$$
Thus, we obtain an equation of a straight line:
\eq[line]{
(t_0-t_1)+x_1(t_1-2t_2)+x_2t_2=0\,.
}
Since our $B$ cannot be a linear function (a linear function cannot
satisfy the boundary condition), we have only one extremal line
passing through a given point. Therefore, this line must coincide
with~\eqref{extrem}, which yields proportionality of the coefficients:
\eq[eq1]{
\frac{dt_0}{t_0-t_1}=\frac{dt_1}{t_1-2t_2}=\frac{dt_2}{t_2}\,.
}

Using the second equality, we express $t_1$ in terms of $t_2:$
\begin{gather*}
t_2dt_1=(t_1-2t_2)dt_2
\\
t_2dt_1-t_1dt_2=-2t_2dt_2
\\
d\big(\frac{t_1}{t_2}\big)=-2\frac{dt_2}{t_2}
\\
\frac{t_1}{t_2}=-2\log|t_2|+2c_1
\\
t_1=-2t_2\log|t_2|+2c_1t_2\,.
\end{gather*}
Now, we use the equality between the first and third terms
in~\eqref{eq1}:
\begin{gather*}
t_2dt_0=(t_0-t_1)dt_2
\\
t_2dt_0-t_0dt_2=-t_1dt_2
\\
\begin{aligned}
d\big(\frac{t_0}{t_2}\big)&=-\frac{t_1}{t_2}\,\frac{dt_2}{t_2}
\\
&=2(\log|t_2|-c_1)d\log|t_2|
\\
&=d(\log^2|t_2|-2c_1\log|t_2|)
\end{aligned}
\\
\frac{t_0}{t_2}=\log^2|t_2|-2c_1\log|t_2|+c_2
\\
t_0=t_2\log^2|t_2|-2c_1t_2\log|t_2|+c_2t_2\,.
\end{gather*}
Dividing~\eqref{line} by $t_2$ gives
$$
\Big(\frac{t_0}{t_2}-\frac{t_1}{t_2}\Big)+
x_1\Big(\frac{t_1}{t_2}-2\Big)+x_2=0.
$$
Plugging into this equality the earlier expressions for $t_1$ and
$t_0,$ we get
\eq[line1]{
\big(\log^2|t_2|-2c_1\log|t_2|+c_2+2\log|t_2|-2c_1\big)
+x_1\big(-2\log|t_2|+2c_1-2\big)+x_2=0\,.
}
From this expression, it is clear that it is convenient to introduce
a new parametrization of our extremal trajectories:
$$
a=\log|t_2|-c_1+1\,.
$$
The equation of the extremal trajectory~\eqref{line1} then takes
the form
$$
a^2-2ax_1+x_2-1+c_2-c_1^2=0\,.
$$
Since
$$
c_1^2+1-c_2=a^2-2ax_1+x_2=(a-x_1)^2+(x_2-x_1^2)\ge0\,,
$$
we can introduce a new positive constant $\delta\df (c_1^2+1-c_2)^{1/2}.$
In this notation~\eqref{line1} becomes
\eq[line2]{
x_2=2ax_1-a^2+\delta^2\,.
}
Note that this is an equation of the line tangent to the parabola
$x_2=x_1^2+\delta^2$ at the point $(a,a^2+\delta^2).$

Now, let us collect everything and write down a formula for $B:$
\begin{align*}
t_0&=(a^2-2a+2-\delta^2)t_2
\\
t_1&=-2(a-1)t_2
\\
t_2&=\pm e^{a+c_1-1}=ce^a
\\
B&=t_0+x_1t_1+x_2t_2
\\
&=(a^2-2a+2-\delta^2-2(a-1)x_1+x_2)t_2
\\
&=2c(1-a-x_1)e^a\,.
\end{align*}
From the equation of the extremal line~\eqref{line2}, we can express
$a$ as a function of $x:$
$$
a=a(x)=x_1\pm\sqrt{\delta^2-x_2+x_1^2}.
$$
Therefore,
$$
B(x)=2c\Big(1\mp\sqrt{\delta^2-x_2+x_1^2}\Big)
\exp\Big\{x_1\pm\sqrt{\delta^2-x_2+x_1^2}\Big\}\,.
$$
We find the constant $c$ from the boundary condition~\eqref{bc},
and we choose the sign by checking the sign of the Hessian. Finally,
we obtain
\eq[bbb]{
B(x)=\frac{1-\sqrt{\delta^2-x_2+x_1^2}}{1-\delta}
\exp\Big\{x_1+\sqrt{\delta^2-x_2+x_1^2}-\delta\Big\}\,.
}

Were we a bit more clever, we could realize from the beginning that
any extremal trajectory must touch the upper boundary tangentially,
because, when splitting the interval, a boundary point $x$
can be split into $x^\pm$ only along the tangential direction.
With that realization all calculations become much simpler.

Indeed, take an extremal line given by~\eqref{line2}. It intersects the
lower boundary $x_2=x_1^2$ at two points $(u,u^2)$ with $u=a\mp\ve.$
Since $B$ has to be linear on the extremal line, we have
$$
B(x)=k(u)(u-x_1)+f(u),
$$
where $f(u)$ is the boundary value of $B.$ We will not specify this
value until the very end of calculation. This will serve to demonstrate
that knowing extremal trajectories in advance allows one to solve some
rather general problems, and not just this specific one.

Let us calculate the partial derivative of $B$ with respect to either
coordinate, say~$x_2:$
$$
t_2=B_{x_2}=\big(k'(u)(x_1-u)-k(u)+f'(u)\big)
\frac{\partial u}{\partial x_2}\,.
$$
Using~\eqref{line2}, we get
$$
\frac{\partial u}{\partial x_2}=\frac{\partial a}{\partial x_2}
=\frac1{2(x_1-a)}
$$
and
$$
t_2=\frac12k'(u)+\frac{\pm\ve k'(u)-k(u)+f'(u)}{2(x_1-a)}\,.
$$
Since $t_2$ has to be constant on the extremal line, we conclude that
$$
t_2=\frac12k'(u)
$$
and the coefficient $k$ satisfies the equation
$$
\mp\ve k'(u)-k(u)+f'(u)=0,
$$
whose general solution is
$$
k(u)=\pm\frac1\ve\int_u^{\pm\infty}e^{-|t-u|/\ve}f'(t)\,dt+
C_\pm e^{\pm u/\ve}\,.
$$
(The same conclusion could be reached by considering $B_{x_1}.$)
Let us not discuss at this point why the constant $C_\pm$ should be
chosen equal to zero, other than say that it is a consequence of our
trying to find the best possible estimate. Rewriting the last formula
for our case, $f(u)=e^u,$ we get
$$
k(u)=\frac1{1\mp\ve}e^u\,,
$$
and, therefore,
\begin{align*}
B(x)&=e^u\Big(\frac{x_1-u}{1\mp\ve}+1\Big)
\\
&=\frac{1\mp\sqrt{\ve^2-x_2+x_1^2}}{1\mp\ve}
\exp\Big\{x_1\pm\sqrt{\ve^2-x_2+x_1^2}\mp\ve\Big\}\,.
\end{align*}
Taking the upper sign throughout, we get~\eqref{bbb} with $\delta=\ve.$
\bigskip

\noindent {\bf Homework assignment.}
{\small \sl \color{blue}
To emphasize the dependence on parameter, let us refer to the
function~\eqref{bbb} as $B(x;\delta).$ Check that $B(x;\ve)$ does
not satisfy the main inequality in the domain $\Omega_\ve,$ but
$B(x;\delta)$ does, provided $\delta\ge\frac3{2\sqrt2}\ve.$}

\section{John--Nirenberg inequality, Part III}

We now prove a geometric result that is crucial to applying the
Bellman function method to the usual, non-dyadic $\BMO$ (recall
that up to this point all discussions were about the dyadic space).
Let $[x,y]$ denote the straight-line segment connecting two points
$x$ and $y$ in the plane. Then we have the following lemma.

\begin{lemma}[Splitting lemma]
\label{splitting}
Fix two positive numbers $\ve,\delta,$ with $\ve<\delta.$ For an
arbitrary interval $I$ and any function $\varphi\in\BMO_\ve(I),$
there exists a splitting $I=I_+\cup I_-$ such that the whole
straight-line segment $[x^{I_-},x^{I_+}]$ is inside $\Omega_\delta$.
Moreover\textup, the parameters of splitting $\alpha_\pm\df|I_{\pm}|/|I|$
are separated form $0$ and $1$ by constants depending on $\ve$ and
$\delta$ only\textup, i.e. uniformly with respect to the choice of
$I$ and $\varphi.$
\end{lemma}

\begin{proof}
Fix an interval $I$ and a function $\varphi\in \BMO_{\ve}(I).$ We now
demonstrate an algorithm to find a splitting $I=I_-\cup I_+$ (i.e.
choose the splitting parameters $\alpha_\pm=|I_\pm|/|I|$) so that
the statement of the lemma holds. For simplicity, put  $x^0=x^I$
and $x^{\pm}=x^{I_\pm}.$

First, we take $\alpha_-=\alpha_+=\frac12$ (see Fig.~\ref{f11}).
\begin{figure}[ht]
\begin{center}
\begin{picture}(200,175)
\thinlines
\put(100,0){\vector(0,1){180}}
\put(10,20){\vector(1,0){180}}
\linethickness{.5pt}
\thicklines
\qbezier[1000](20,82)(100,-42)(180,82)
\qbezier[1000](20,162)(100,15)(180,162)
\qbezier[1000](20,168)(100,15)(180,168)
\qbezier[1000](28,77)(28,77)(138,101)
\put(30,69){\footnotesize $x^-$}
\put(83,80){\footnotesize $x^0$}
\put(138,93){\footnotesize $\xi$}
\put(28,77){\circle*{2}}
\put(83,89){\circle*{2}}
\put(138,101){\circle*{2}}
\put(183,80){\footnotesize $x_2=x_1^2$}
\put(183,160){\footnotesize $x_2=x_1^2+\ve^2$}
\put(183,175){\footnotesize $x_2=x_1^2+\delta^2$}
\end{picture}
\caption{The initial splitting: $\alpha_-=\alpha_+=\frac12,~\xi=x^+.$}
\label{f11}
\end{center}
\end{figure}
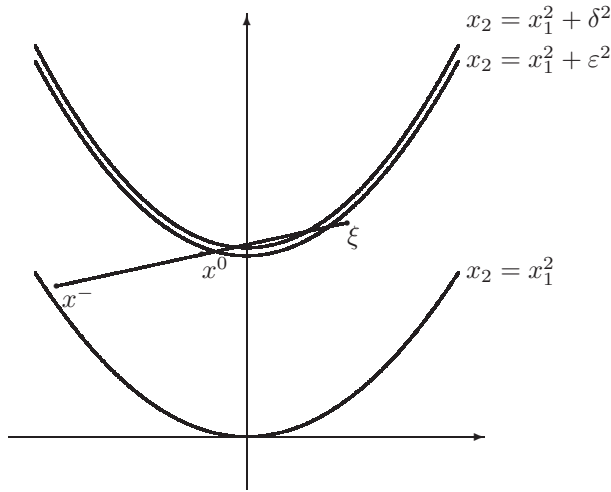
If the whole segment $[x^-,x^+]$ is in $\Omega_{\delta},$ we fix this
splitting. Assuming it is not the case, there exists a point $x$ on
this segment with $x_2-x_1^2>\delta^2.$ Observe that only one of the
segments, either $[x^-,x^0]$ or $[x^+,x^0],$ contains such points.
Denote the corresponding endpoint ($x^-$ or $x^+$) by $\xi$ and define
a function $\rho$ by
$$
\rho(\alpha_+)=\max_{x\in [x^-\!,\,x^+]}\{x_2-x_1^2\}=
\max_{x\in [\xi,\,x^0]}\{x_2-x_1^2\}.
$$
By assumption, $\rho\left(\frac12\right)>\delta^2.$ We will now change
$\alpha_+$ so that $\xi$ approaches $x^0,$ i.e. we will increase
$\alpha_+$ if $\xi=x^+$ and decrease it if $\xi=x^-.$ We stop when
$\rho(\alpha_+)=\delta^2$ and fix that splitting. It remains to check
that such a moment occurs and that the corresponding $\alpha_+$ is
separated from 0 and 1.

Without loss of generality, assume that $\xi=x^+.$ Since the function
$x^+(\alpha_+)$ is continuous on the interval $(0,1]$ and $x^+(1)=x^0,$
$\rho$ is continuous on $[\frac12,1].$ We have
$\rho\left(\frac12\right)>\delta^2$ and we also know that $\rho(1)\le
\ve^2<\delta^2$ (because $x^0\in \Omega_\ve$). Therefore, there is a
point $\alpha_+\in\left[\frac12,1\right]$ with $\rho(\alpha_+)=\delta^2$
(Fig.~\ref{f12}).

Having just proved that the desired point exists, we need to check that
the corresponding $\alpha_+$ is not too close to 0 or 1. If $\xi=x^+,$
we have $\alpha_+>\frac12$ and $\xi_1-x_1^0=x_1^+-x_1^0=
\alpha_-(x_1^+-x_1^-).$ Similarly, if $\xi=x^-,$ we have
$\alpha_->\frac12$ and $\xi_1-x_1^0=x_1^--x_1^0=\alpha_+(x_1^--x_1^+).$
Thus, $|\xi_1-x_1^0|=\min\{\alpha_\pm\}|x_1^--x_1^+|.$
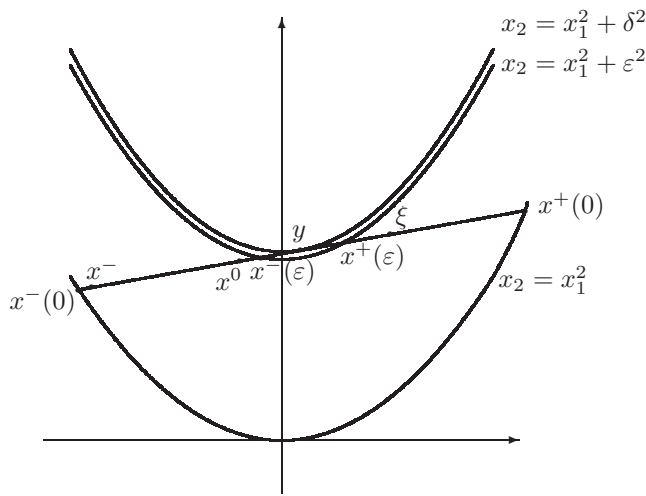
\begin{figure}[ht]
\begin{center}
\begin{picture}(200,175)
\thinlines
\put(100,0){\vector(0,1){180}}
\put(10,20){\vector(1,0){180}}
\linethickness{.5pt}
\thicklines
\qbezier[1000](20,82)(100,-42)(180,82)
\qbezier[1000](20,162)(100,15)(180,162)
\qbezier[1000](20,168)(100,15)(180,168)
\qbezier[1000](23,77)(23,77)(192,107)
\qbezier[200](180,82)(190,100)(193,110)
\put(23,77){\circle*{2}}
\put(28,78){\circle*{2}}
\put(141,98){\circle*{2}}
\put(84,88){\circle*{2}}
\put(92,89){\circle*{2}}
\put(107,92){\circle*{2}}
\put(124,95){\circle*{2}}
\put(192,107){\circle*{2}}
\put(182,78){\footnotesize $x_2=x_1^2$}
\put(183,160){\footnotesize $x_2=x_1^2+\ve^2$}
\put(183,175){\footnotesize $x_2=x_1^2+\delta^2$}
\put(26,81){\footnotesize $x^-$}
\put(75,77){\footnotesize $x^0$}
\put(88,81){\footnotesize $x^-(\ve)$}
\put(104,97){\footnotesize $y$}
\put(122,87){\footnotesize $x^+(\ve)$}
\put(143,102){\footnotesize $\xi$}
\put(197,106){\footnotesize $x^+(0)$}
\put(-3,70){\footnotesize $x^-(0)$}
\end{picture}
\caption{The stopping time: $[x^-\!,\,\xi]$ is tangent to the parabola $x_2=x_1^2+\ve^2.$}
\label{f12}
\end{center}
\end{figure}
For the stopping value of $\alpha_+,$ the straight line through the points
$x^-,x^+,$ and $x^0$ is tangent to the parabola $x_2=x_1^2+\delta^2$ at
some point $y.$ The equation of this line is, therefore,
$x_2=2x_1y_1-y_1^2+\delta^2.$ The line intersects the graph of
$x_2=x_1^2+s^2$ at the points
$$
x^{\pm}(s)=\left(y_1\pm\sqrt{\delta^2-s^2},\,
y_2\pm2y_1\sqrt{\delta^2-s^2}\right)\,.
$$
Let us focus on the points $x^{\pm}(0)$ and $x^{\pm}(\ve).$ We have
$$
[x^-(\ve),\,x^+(\ve)]\subset[x^0\!,\,\xi]\subset[x^-\!,\,x^+]
\subset[x^-(0),\,x^+(0)]
$$
and, therefore,
$$
\begin{array}{lll}
2\sqrt{\delta^2-\ve^2}&=&|x_1^+(\ve)-x_1^-(\ve)|\le|x_1^0-\xi_1|
=\min\{\alpha_\pm\}|x_1^+-x_1^-|
\\&&\\
&\le&\min\{\alpha_\pm\}|x_1^+(0)-x_1^-(0)|=\min\{\alpha_\pm\}2\delta,
\end{array}
$$
which implies
$$
\sqrt{1-\left(\frac{\ve}{\delta}\right)^2}\le
\alpha_+\le 1-\sqrt{1-\left(\frac{\ve}{\delta}\right)^2}.
$$
As promised, this estimate does not depend on $\varphi$ or $I.$
\end{proof}

From now on, we shall consider not the dyadic Bellman function $\Bel$,
but the ``true'' one:
$$
\Bel(x;\ve)\;\;\df\sup_{\varphi\in\BMO_\ve(J)}\left\{\av{e^\varphi}J
\colon\av{\varphi}J=x_1,\;\av{\varphi^2}J=x_2\right\}.
$$
The test functions now run over the $\ve$-ball of the non-dyadic \BMO.

Using the splitting lemma, we are able to make the Bellman induction
work in the non-dyadic case.

\begin{lemma}[Bellman induction]
If $B$ is a continuous\textup, locally concave function on the
domain $\Omega_\delta,$ satisfying the boundary
condition~\eqref{bcJN}\textup, then $\Bel(x;\ve)\le B(x)$ for
all $\ve<\delta.$
\end{lemma}

\begin{proof}
Fix a function $\varphi\in\BMO_\ve(J)$. By the splitting lemma we
can split every subinterval $I\subset J,$ in such a way that the
segment $[x^{I_-},x^{I_+}]$ is inside $\Omega_\delta$. Since $B$
is locally concave, we have
$$
|I|B(x^I)\ge|I_+|B(x^{I_+})+|I_-|B(x^{I_-})
$$
for any such splitting. Now we can repeat, word for word, the
arguments used in the dyadic case. If $\D_n$ is the set of intervals
of $n$-th generation, then
$$
|J|B(x^J)\ge|J_+|B(x^{J_+})+|J_-|B(x^{J_-})
\ge\sum_{I\in\D_n}|I|B(x^I)=\int_JB(x^{(n)}(s))\,ds\,,
$$
where $x^{(n)}(s)=x^I,$ when $s\in I,$ $I\in\D_n.$ By the Lebesgue
differentiation theorem we have $x^{(n)}(s)\to(\varphi(s),\varphi^2(s))$
almost everywhere. (We have used here the fact that we split the
intervals so that all coefficients $\alpha_\pm$ are uniformly separated
from~$0$ and~$1,$ and, therefore, $\max\{|I|\colon I\in\D_n\}\to0$ as
$n\to\infty).$ Now, we can pass to the limit in this inequality as
$n\to\infty$. Again, first we assume $\varphi$ to be bounded and, by
the Lebesgue dominated convergence theorem, pass to the limit in the
integral using the boundary condition~\eqref{bcJN}:
$$
|J|B(x^J)\ge\int_JB(\varphi(s),\varphi^2(s))\,ds=
\int_Je^{\varphi(s)}ds=|J|\av{e^\varphi}J\,.
$$
Then using the cut-off approximation, we get the same inequality
for an arbitrary $\varphi\in\BMO_\ve(J).$
\end{proof}

\begin{cor}
$$
\Bel(x;\ve)\le B(x;\delta)\qquad \ve<\delta<1\,.
$$
\end{cor}

\begin{proof}
The function $B(x;\delta)$ was constructed as a locally concave function
satisfying boundary condition~\eqref{bcJN}.
\end{proof}
\begin{cor}
\eq[nn]{
\Bel(x;\ve)\le B(x;\ve)\,.
}
\end{cor}
\begin{proof}
Since the function $B(x;\delta)$ is continuous with respect to the
parameter $\delta\in(0,1),$ we can pass to the limit $\delta\to\ve$
in the preceding corollary.
\end{proof}

Now, we would like to prove the inequality converse to~\eqref{nn}. To
this end, for every point $x$ of $\Omega_\ve$ we construct a test
function $\varphi$ on any interval with $\BMO$ norm $\ve,$ satisfying
$\av{e^\varphi}{}=B(x;\ve),$ and such that its Bellman point is $x$
(let us call such a function an {\it optimizer} for the point $x$).
This would imply the inequality $\Bel(x;\ve)\ge B(x;\ve).$

First, we construct an optimizer $\varphi_0$ for the point $(0,\ve^2).$
Without loss of generality, we can work on $I\df[0,1].$ Note that the
function $\varphi_a\df\varphi_0+a$ will then be an optimizer for the
point $(a,a^2+\ve^2).$ Indeed, $\varphi_a$ has the same norm as
$\varphi_0,$ and if
$$
\av{e^{\varphi_0}}{}=B(0,\ve^2;\ve)=\frac{e^{-\ve}}{1-\ve}\,,
$$
then
$$
\av{e^{\varphi_a}}{}=\frac{e^{a-\ve}}{1-\ve}=B(a,a^2+\ve^2;\ve)\,.
$$
The point $(0,\ve^2)$ is on the extremal line starting at $(-\ve,\ve^2).$
To keep equality on each step of the Bellman induction, when we split
$I$ into two subintervals $I_-$ and $I_+,$ the segment $[x^-\!,\,x^+]$
has to be contained in the extremal line along which our function $B$
is linear. Since $x$ is a convex combination of $x^-$  and $x^+,$ one
of these points, say $x^+,$ has to be to the right of $x.$ However,
the extremal line ends at $x=(0,\ve^2),$ and so there seems to be
nowhere to place that point. We circumvent this difficulty by placing
$x^+$ infinitesimally close to $x$ and using an approximation procedure.
Where should $x^-$ be placed? We already know optimizers for points on
the lower boundary $x_2=x_1^2,$ since the only test function there are
constants. Thus, it is convenient to put $x^-$ there. Therefore, we set
$$
x^-=(-\ve,\ve^2)\qquad\text{and}\qquad x^+=(\Delta\ve,\ve^2)\,,
$$
for small $\Delta.$ To get these two points, we have to split $I$ in
proportion $1\!:\!\Delta,$ that is we take $I_+=[0,\frac1{1+\Delta}]$
and $I_-=[\frac1{1+\Delta},1].$
\begin{figure}[ht]
\begin{center}
\begin{picture}(300,50)
\thicklines
\put(60,20){\line(0,1){10}}
\put(220,20){\line(0,1){10}}
\put(240,20){\line(0,1){10}}
\put(60,25){\line(1,0){180}}
\put(145,-5){$I_+$}
\put(227,-5){$I_-$}
\put(57,10){$0$}
\put(205,10){\footnotesize$\frac1{1+\Delta}$}
\put(237,10){$1$}
\put(10,35){\footnotesize$\varphi_0(t)\approx$}
\put(110,35){\footnotesize$\varphi_{\Delta\ve}((1+\Delta)t)$}
\put(223,35){\footnotesize$-\ve$}
\end{picture}
\label{f18}
\end{center}
\end{figure}
To get the point $x^-,$ we have to put $\varphi_0(t)=-\ve$ on $I_-.$
On $I_+,$ we put a function corresponding not to the point $x^+,$ but to the point
$(\Delta\ve,(1+\Delta^2)\ve)$ on the upper boundary, which is close to $x^+$ (the
distance between these two points is of order $\Delta^2).$ For such a point the
extremal function is $\varphi_{\Delta\ve}(t)=\varphi_0(t)+\Delta\ve.$ Therefore,
this function, when properly rescaled, can be placed on $I_+.$ As a result, we obtain
$$
\varphi_0(t)\approx\varphi_0\big((1+\Delta)t\big)+\Delta\ve
\approx\varphi_0(t)+\varphi_0'(t)\Delta t+\Delta\ve\,,
$$
which yields
$$
\varphi'_0(t)=-\frac\ve{t}\,.
$$
Taking into account the boundary condition $\varphi_0(1)=-\ve,$ we get
$$
\varphi_0(t)=\ve\log\frac1t-\ve\,.
$$

Let us check that we have found what we need:
$$
\av{e^{\varphi_0}}{[0,1]}=\int_0^1\!\!e^{-\ve}\frac{dt}{t^\ve}
=\frac{e^{-\ve}}{1-\ve}=B(0,\ve^2;\ve)\,.
$$

It easy now to get an extremal function for an arbitrary point $x$ in $\Omega_\ve.$
First of all, we draw the extremal line through $x$. It touches the upper boundary
at the point $(a,a^2+\ve^2)$ with $a=x_1+\sqrt{\ve^2-x_2+x_1^2}$ and intersects the
lower boundary at the point $(u,u^2)$ with $u=a-\ve.$ Now, we split the interval $[0,1]$
in proportion $(x_1-u)\!:\!(a-x_1)$ and concatenate the two known optimizers, $\varphi=u$
for the $x^-=(u,u^2)$ and $\varphi=\varphi_a$ for $x^+=(a,a^2+\ve^2)$. This gives the
following function:
$$
\varphi(t)=
\begin{cases}
\ve\log\frac{x_1-u}t+u&0\le t\le x_1-u
\\
\quad u&x_1-u\le t\le1
\end{cases},
\quad\text{where}\quad u=x_1+\sqrt{\ve^2-x_2+x_1^2}-\ve\,.
$$
This is a function from $\BMO_\ve$ satisfying the required property $\av{e^\varphi}{[0,1]}
=B(x;\ve)$ (see the homework assignment below).

This completes the proof of the following theorem
\begin{thmnonum}
If $\ve<1,$ then
$$
\Bel(x;\ve)=\frac{1-\sqrt{\ve^2-x_2+x_1^2}}{1-\ve}
\exp\Big\{x_1+\sqrt{\ve^2-x_2+x_1^2}-\ve\Big\}\,;
$$
if $\ve\ge1,$ then $\Bel(x;\ve)=\infty.$
\end{thmnonum}

Indeed, the second statement can be verified by the same extremal function $\varphi,$
because $e^\varphi$ is not summable on $[0,1]$ for $\ve\ge1.$
\bigskip

The first proof of the theorem above appeared in~\cite{Sl} and~\cite{Va}; a complete proof
of this result together with the estimate from below (i.e. the lower Bellman function) and consideration of the dyadic version of the
problem can be found in~\cite{SlVa1} (the online version of this paper is~\cite{SlVa2}).
\bigskip

\noindent {\bf Homework assignment 1.}
{\small \sl \color{blue}
Verify the following properties of the extremal function~$\varphi:$
\begin{itemize}
\item$\av\varphi{[0,1]}=x_1;$
\item$\av{\varphi^2}{[0,1]}=x_2;$
\item$\av{e^\varphi}{[0,1]}=B(x_1,x_2;\ve);$
\item $\varphi\in\BMO_\ve.$
\footnote{{\it Hint}: Due to the cut-off lemma (Lemma~\ref{norm}),
it is sufficient to check that $\log t\in\BMO_1,$ which follows from
$\av{\log^2t}{[c,d]}-\av{\log t}{[c,d]}^2 =1-\frac{cd}{(d-c)^2}\Big(\log\frac dc\Big)^2.$}
\end{itemize}}
\noindent {\bf Homework assignment 2.}
{\small \sl \color{blue}
Recall that we also obtained a second solution,
$$
b(x;\ve)=\frac{1+\sqrt{\ve^2-x_2+x_1^2}}{1+\ve}\;
\exp\!\Big\{x_1-\sqrt{\ve^2-x_2+x_1^2}+\ve\Big\}\,.
$$
Check that this is the solution of the following extremal problem:
$$
\bel(x;\ve)\;\;\df\inf_{\varphi\in\BMO_\ve(J)}\left\{\av{e^\varphi}J
\colon\av{\varphi}J=x_1,\;\av{\varphi^2}J=x_2\right\},
$$
that is check that the Bellman induction works and construct an extremal function
for every $x\in\Omega_\ve.$
}

\section{Dyadic maximal operator}

Let us define the dyadic maximal operator on the set of positive
locally summable functions $w,$ as follows:
$$
(Mw)(t)\;\df\sup_{I\in\D_{\R},\,t\in I}\av wI\,.
$$

We would like to estimate the norm of $M$ as an operator acting from
$L^2(\R)$ to $L^2(\R).$ Even though the operator is defined
on the whole line, we first localize its action to a fixed dyadic interval
$J;$ we will pass to all of $\R$ at the end. Thus, we are looking for the function
$$
\Bel(x_1,x_2;L)\df\sup_{w\ge0}\left\{\av{(Mw)^2}J\colon\;
\av wJ=x_1,\;\av{w^2}J=x_2,\;\sup_{I\supset J,\;I\in\mathcal{D}_\R}\av wI=L\right\}\,.
$$
We need the ``external'' parameter $L$ because $M$ is not truly local: the value
of $Mw$ on an interval $J$ depends not only on the behavior of $w$ on $J,$ but
also on that on the whole line $\R.$ The function $\Bel$ depends on three variables,
and each of them can change when we split the interval of definition. Nevertheless,
we will consider $L$ as a parameter. The reason will become clear a bit later.

As before, $\Bel$ does not depend on $J.$  Its domain is
$$
\Omega\df\left\{(x_1,x_2;L)\colon 0<x_1\le L,\;x_1^2\le x_2\right\},
$$
or, if we consider $L$ as a fixed parameter,
$$
\Omega_L\df\left\{(x_1,x_2)\colon 0<x_1\le L,\;x_1^2\le x_2\right\}.
$$
\begin{lemma}[Main inequality]
Take $(x;L)\in\Omega$ and let the points $(x^\pm;L^\pm)\in\Omega$ be such
that $x=(x^++x^-)/2$ and $L^\pm=\max\{x_1^\pm,L\}.$ Then the following
inequality holds\textup:
\eq[mainMF]{
\Bel(x;L)\ge\frac{\Bel(x^+;L^+)+\Bel(x^-;L^-)}2\,.
}
\end{lemma}
\begin{proof}
The proof is now standard. Fixing an interval $J$ and a small
number $\eta>0,$ we take a pair of functions $w^\pm$ such that
$$
\Bel(x^\pm;L^\pm)\ge\av{(Mw^\pm)^2}{J_\pm}-\eta
$$
and set
$$
w(t)=
\begin{cases}
w^\pm(t),&\text{if }t\in J_\pm,
\\
\ L,&\text{if }t\notin J\,.
\end{cases}
$$
Then $w$ is a test function corresponding to the Bellman point $(x;L)$
with $(Mw)(t)=(Mw^\pm)(t)$ for $t\in J_\pm.$ Therefore,
\begin{align*}
\Bel(x;L)\ge\av{(Mw)^2}J&=\frac12\big(\av{(Mw^+)^2}{J_+}+
\av{(Mw^-)^2}{J_-}\big)
\\
&\ge\frac12\big(\Bel(x^+;L^+)+\Bel(x^-;L^-)\big)-\eta\,,
\end{align*}
which proves the lemma.
\end{proof}

\begin{cor}[Concavity]
For a fixed $L,$ the function $\Bel$ is concave on $\Omega_L.$
\end{cor}

\begin{proof}
For any pair $x^\pm\in\Omega_L$, we have $L^\pm=L$ and~\eqref{mainMF}
becomes the usual concavity condition.
\end{proof}

\begin{cor}[Boundary condition]
If the function $\Bel$ is sufficiently smooth\textup, then
\eq[bcN]{
\frac{\partial\Bel}{\partial L}(x;x_1)=0\,.
}
\end{cor}

\begin{proof}
First of all, we note that the definition of $\Bel$ immediately yields
the inequality $\frac{\partial\Bel}{\partial L}\ge0.$ Now, take an
arbitrary point $x$ on the boundary $x_1=L$ and a pair $x^\pm$ such
that $x=(x^++x^-)/2.$ Let $\Delta_k=(x^+_k-x^-_k)/2, k=1,2,$ and assume,
without loss of generality, that $\Delta_1>0.$ Then
$x^\pm_k=x_k\pm\Delta_k$ and $x^-_1<x_1=L<x^+_1;$ therefore, $L^+=x^+_1$
and $L^-=L.$ Writing the main inequality up to the terms of first
order in $\Delta,$ we get
\begin{align*}
0&\le\Bel(x;L)-\frac12\bigl(\Bel(x^+;L^+)+\Bel(x^-;L^-)\bigr)
\\
&=\Bel(x_1,x_2;x_1)-\frac12
\bigl(\Bel(x_1+\Delta_1,x_2+\Delta_2;x_1+\Delta_1)+
\Bel(x_1-\Delta_1,x_2-\Delta_2;x_1)\bigr)
\\
&\approx\Bel(x;x_1)-\frac12\bigl(\Bel(x;x_1)+\Bel_{x_1}\Delta_1+
\Bel_{x_2}\Delta_2+\Bel_L\Delta_1+\Bel(x;x_1)-\Bel_{x_1}\Delta_1-
\Bel_{x_2}\Delta_2\bigr)
\\
&=-\frac12\Bel_L(x;x_1)\Delta_1\,.
\end{align*}
Since $\Bel_L(x;x_1)\ge0,$ the last inequality is possible only if
$\Bel_L(x;x_1)=0.$
\end{proof}

\begin{lemma}[Homogeneity]
If the function $\Bel$ is sufficiently smooth, then
\eq[homogMF]{
\Bel(x;L)=\frac12x_1\Bel_{x_1}
+x_2\Bel_{x_2}
+\frac12L\,\Bel_L\,.
}
\end{lemma}
\begin{proof}
As before, together with a test function $w$ we consider the function
$\tilde w=\tau w$ for $\tau>0.$ Comparing the Bellman functions at the
corresponding Bellman points gives us the equality
$$
\Bel(\tau x_1,\tau^2x_2;\tau L)=\tau^2\Bel(x_1,x_2;L).
$$
Differentiating this identity with respect to $\tau$ at the point $\tau=1$
proves the lemma.
\end{proof}

Before we start looking for a Bellman candidate, let us state one more
boundary condition --- in fact, the principal one.
\begin{lemma}[Boundary condition]
\eq[bcD]{
\Bel(u,u^2;L)=L^2\,.
}
\end{lemma}
\begin{proof}
The only test function corresponding to the point $x=(u,u^2)$ is the
function identically equal to $u$ on the interval $J.$ Hence,
$Mw$ is identically $L$ on this interval.
\end{proof}
\begin{remark}
Note that the boundary $x_1=0$ is not accessible\textup, that is it does not belong
to the domain\textup: if $x_1=0,$ $w$ must be identically zero on $J,$ which means that
$x_2=0.$ Therefore\textup, no boundary condition can be stated on that boundary.
\end{remark}

We are now ready to search for a Bellman candidate. To this end, we will, as
before, solve a \ma boundary value problem.
The arguments why we are looking for a solution of the Monge--Amp\`ere
equation are the same as before: the concavity condition forces
us to look for a function whose Hessian is negative and the optimality
condition requires the Hessian to be degenerate.

Again, we are looking for a solution in the form
$$
B(x)=t_0+x_1t_1+x_2t_2
$$
that is linear along extremal trajectories given by
$$
dt_0+x_1dt_1+x_2dt_2=0.
$$
Let us parameterize the extremal lines by the first coordinate of their
points of intersection with the boundary $x_2=x_1^2.$ Since the boundary
$x_1=0$ is not accessible, such an extremal line can either be vertical
(i.e. parallel to the $x_2$-axis) or slant to the right, in which case
it intersects the boundary $x_1=L$ at a point, say, $(L,v).$ ({\label{hw} \bf A homework
question}: {\sl \color{blue} why can an extremal line never connect two
points of the boundary $x_2=x_1^2?$})

\begin{figure}[ht]
\begin{center}
\begin{picture}(200,200)
\thinlines
\put(10,0){\vector(0,1){180}}
\put(0,10){\vector(1,0){180}}
\linethickness{.5pt}
\thicklines
\qbezier[1000](10,10)(90,10)(170,70)
\qbezier[500](150,56)(150,270)(150,56)
\qbezier[1000](95,27)(205,253)(95,27)
\multiput(150,10)(0,3){16}{\circle*{1}}
\multiput(95,10)(0,3){7}{\circle*{1}}
\multiput(10,27)(3,0){29}{\circle*{1}}
\multiput(10,140)(3,0){47}{\circle*{1}}
\put(185,5){\footnotesize $x_1$}
\put(5,185){\footnotesize $x_2$}
\put(148,0){\footnotesize $L$}
\put(0,25){\footnotesize $u^2$}
\put(2,138){\footnotesize $v$}
\put(93,0){\footnotesize $u$}
\put(95,27){\circle*{2}}
\put(150,140){\circle*{2}}
\put(175,75){\footnotesize $x_2=x_1^2$}
\end{picture}
\caption{The extremal trajectory passing through $(u,u^2)$ and $(L,v)$}
\end{center}
\end{figure}
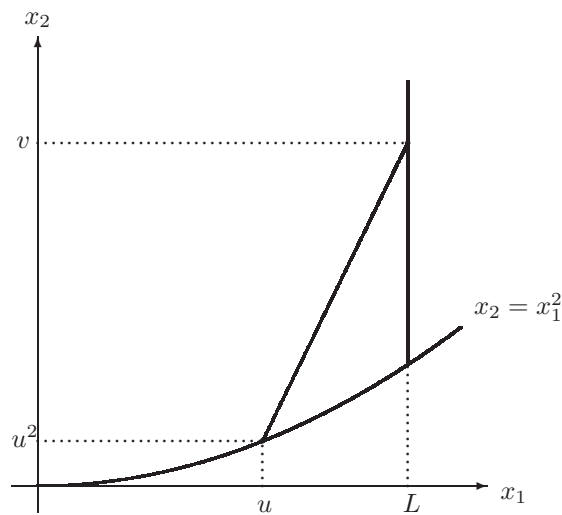

The former case is very simple. Since $B$ is linear on each vertical
line and satisfies the boundary condition~\eqref{bcD}, it has the form
$$
B(x;L)=k(x_1,L)(x_2-x_1^2)+L^2.
$$
Since $B_{x_2x_2}=0$ and the matrix $B_{x_ix_j}$ must be non-positive,
we have $B_{x_1x_2}=k_{x_1}=0,$ i.e. $k$ does not depend on $x_1,$
$k=k(L).$ Now we use the second boundary condition $B_L(x_1,x_2;x_1)=0,$
which turns into $k'(x_1)(x_2-x_1^2)+2x_1=0.$ The last equation has
no solution, therefore this case is impossible, at least in the whole domain $\Omega_L.$

Consider the latter case, when the extremal line goes from the bottom boundary to
the right boundary, as shown in the picture. The boundary condition~\eqref{bcD}
on the bottom boundary gives us
\eq[mf1]{
t_0+ut_1+u^2t_2=L^2,
}
and the condition~ \eqref{bcN} on the right boundary, together with~\eqref{homogMF}, yields
\eq[mf2]{
t_0+Lt_1+vt_2=\frac12Lt_1+vt_2.
}
From the last equation, we get
\eq[mf3]{
t_0+\frac12Lt_1=0.
}
Now we differentiate~\eqref{mf1},
$$
(dt_0+udt_1+u^2dt_2)+(t_1+2ut_2)du=0,
$$
and use the fact that the point $(u,u^2)$ is on the trajectory, i.e.
$dt_0+udt_1+u^2dt_2=0.$ Thus,
$$
(t_1+2ut_2)du=0.
$$
This equations gives us two possibilities: either $u=\const,$ producing a family
of trajectories all passing through the point $(u,u^2),$ or $t_1+2ut_2=0.$
The first possibility cannot give a foliation of the whole $\Omega_L,$ since it
would result in trajectories connecting two points of the bottom boundary (an
impossibility, by the earlier \hyperref[hw]{\color{blue} homework question}). Therefore, let
us consider the second possibility, i.e.
\eq[mf4]{
t_1+2ut_2=0.
}
Solving the system of three linear equations,~\eqref{mf1}, \eqref{mf3},
and~\eqref{mf4}, of three variables $t_0,\;t_1,$ and $t_2$, we obtain:
\begin{align*}
t_0&=\frac{L^3}{L-u}\,,&\qquad t_0'&=\frac{L^3}{(L-u)^2}\,,
\\
t_1&=-\frac{2L^2}{L-u}\,,&\qquad t_1'&=-\frac{2L^2}{(L-u)^2}\,,
\\
t_2&=\frac{L^2}{u(L-u)}\,,&\qquad t_2'&=-\frac{L^2(L-2u)}{u^2(L-u)^2}\,.
\end{align*}
 Now we can plug the derivatives of $t_i$ into the equation of
extremal trajectories \\ $dt_0+x_1dt_1+x_2dt_2=0:$
\eq[mf5]{
\frac{L^3}{(L-u)^2}-x_1\frac{2L^2}{(L-u)^2}-
x_2\frac{L^2(L-2u)}{u^2(L-u)^2}=0\,,
}
or
$$
x_2=\frac{2u^2}{2u-L}\Big(x_1-\frac L2\Big)\,.
$$
We see that this is a ``fan'' of lines passing through the point $(L/2,0).$
However, those elements of this fan that intersect the ``forbidden'' boundary $x_1=0$
cannot be extremal trajectories. Therefore, the acceptable lines foliate
not the whole domain $\Omega_L,$ but only the sub-domain $x_1\ge L/2.$ To foliate
the rest, we return to considering vertical lines. Earlier, we have refused this type
of trajectories for the whole domain $\Omega_L,$ since the foliation so produced would
not give a function satisfying the boundary condition on the line $x_1=L.$ However,
such trajectories are perfectly suited for foliating the sub-domain $x_1\le L/2,$
especially because the boundary of the two sub-domains, the vertical line $x_1=L/2,$
fits as an element of both foliations. On this line, we have
$$
t_0=2L^2,\qquad t_1=-4L,\qquad t_2=4,
$$
and so
$$
B(L/2,x_2;L)=2L^2-4L(L/2)+4x_2=4x_2.
$$
As we have seen, the Bellman candidate on the vertical trajectories
must be of the form
$$
B(x;L)=k(L)(x_2-x_1^2)+L^2.
$$
To get $B=4x_2$ on the line $x_1=L/2,$ we have to take $k(L)=4,$ which gives the
following Bellman candidate in the left half of $\Omega_L:$
$$
B(x;L)=4(x_2-x_1^2)+L^2.
$$

To have an explicit formula for the Bellman candidate in the right
half of $\Omega_L,$ we need an expression for $u,$ which
we find solving equation~\eqref{mf5}:
$$
u=\frac{\sqrt{x_2}L}{\sqrt{x_2}+\sqrt{x_2-L(2x_1-L)}}\,.
$$
This yields
$$
B(x;L)=\big(\sqrt{x_2}+\sqrt{x_2-L(2x_1-L)}\big)^2.
$$
Finally, our Bellman candidate in $\Omega$ is given by
\eq[BFMF]{
B(x;L)=
\begin{cases}
\qquad4(x_2-x_1^2)+L^2,&0<x_1\le\frac L2,\ x_2\ge x_1^2,
\\
\big(\sqrt{x_2}+\sqrt{x_2-L(2x_1-L)}\big)^2,
&\frac L2\le x_1\le L,\ x_2\ge x_1^2.\rule{0pt}{18pt}
\end{cases}
}
\bigskip

Now, we start proving that the Bellman candidate just found is indeed
the Bellman function of our problem.

\begin{lemma}
The function defined by~\eqref{BFMF} satisfies the main
inequality~\eqref{mainMF}.
\end{lemma}

\begin{proof}
Let us define a new function $\tilde B$ in the domain \\$\tilde\Omega\df
\{x=(x_1,x_2)\colon\,x_1>0,\,x_2\ge x_1^2\}:$
$$
\tilde B(x;L)\df
\begin{cases}
B(x;L),&0<x_1\le L,\ x_2\ge x_1^2,
\\
B(x;x_1),
&x_1\ge L,\ x_2\ge x_1^2,\rule{0pt}{18pt}
\end{cases}
$$
or
$$
\tilde B(x;L)=
\begin{cases}
4(x_2-x_1^2)+L^2,&0<x_1\le\frac L2,\ x_2\ge x_1^2,
\\
\big(\sqrt{x_2}+\sqrt{x_2-L(2x_1-L)}\,\big)^2,
&\frac L2\le x_1\le L,\ x_2\ge x_1^2.\rule{0pt}{18pt}
\\
\big(\sqrt{x_2}+\sqrt{x_2-x_1^2)}\,\big)^2,
&x_1\ge L,\ x_2\ge x_1^2.\rule{0pt}{18pt}
\end{cases}
$$
Let us calculate the first partial derivatives:
$$
\tilde B_{x_1}(x;L)=
\begin{cases}
-8Lx_1,&0<x_1\le\frac L2,
\\
\ds-2L\Big(1+\frac{\sqrt{x_2}}{\sqrt{x_2-L(2x_1-L)}}\Big),
&\frac L2\le x_1\le L,\rule{0pt}{25pt}
\\
\ds-2x_1\Big(1+\frac{\sqrt{x_2}}{\sqrt{x_2-x_1^2}}\Big),
&x_1\ge L,\ x_2\ge x_1^2.
\end{cases}
$$

$$
\tilde B_{x_2}(x;L)=
\begin{cases}
4,&0<x_1\le\frac L2,
\\
\ds\frac{\big(\sqrt{x_2}+\sqrt{x_2-L(2x_1-L)}\big)^2}
{\sqrt{x_2}\sqrt{x_2-L(2x_1-L)}},
&\frac L2\le x_1\le L,\rule{0pt}{25pt}
\\
\ds\frac{\big(\sqrt{x_2}+\sqrt{x_2-x_1^2}\big)^2}
{\sqrt{x_2}\sqrt{x_2-x_1^2}},
&\frac L2\le x_1\le L.
\end{cases}
$$
From these expressions we see that our function $\tilde B$ is
$C^1$-smooth. Since the second derivative
$$
\tilde B_{x_1x_1}(x;L)=
\begin{cases}
-8L,&0<x_1\le\frac L2,
\\
\ds-\frac{2L^2\sqrt{x_2}}{\big(x_2-L(2x_1-L)\big)^{3/2}},
&\frac L2\le x_1\le L,\rule{0pt}{25pt}
\\
\ds-2\Big(\frac{\sqrt{x_2}}{\sqrt{x_2-x_1^2}}\Big)^3-2,
&x_1\ge L,\ x_2\ge x_1^2,
\end{cases}
$$
is negative, one can check the concavity of $\tilde B$ in the domain
$\Omega_+\!\!\df\!\!\{x: x_1>0, x_2\ge x_1^2\}$ by verifying that the
determinant of the Hessian matrix is non-negative. We know that this
determinant is zero in $\Omega_L$ and, therefore, need to calculate
the second derivatives of $\tilde B$ only in the domain $x_1>L,$
where $\tilde B(x_1,x_2)=B(x_1,x_2;x_1).$ In this domain, we have
\begin{align*}
\tilde B_{x_1x_2}&=\frac{x_1^3}{\sqrt{x_2}}\big(x_2-x_1^2\big)^{-3/2},
\\
\tilde B_{x_2x_2}&=-\frac{x_1^4}{2x_2^{3/2}}\big(x_2-x_1^2\big)^{-3/2},
\end{align*}
which yields
$$
\tilde B_{x_1x_1}\tilde B_{x_2x_2}-\tilde B_{x_1x_2}^2
=\frac{x_1^4}{x_2(x_2-x_1^2)^2_{\phantom2}}+
\frac{x_1^4}{x_2^{3/2}(x_2-x_1^2)^{3/2}_{\phantom2}}>0\,.
$$
The concavity just proved immediately implies~\eqref{mainMF}. Indeed,
we have proved that the function $\tilde B$ is locally concave in each
sub-domain of $\Omega_+,$ as well as $C^1$-smooth in the whole domain;
therefore, it is concave everywhere in $\Omega_+.$ Furthermore,
relation~\eqref{mainMF} is a special case of the concavity condition
on the function $\tilde B:$ $x^-$ and $x$ are in the sub-domain
$\Omega_L,$ while $x^+$ may be either in $\Omega_L$ or the sub-domain $x_1>L.$
\end{proof}

\begin{lemma}[Bellman induction]
For any continuous function $B$ satisfying the main
inequality~\eqref{mainMF} and the boundary
condition~\eqref{bcD}\textup, we have
$$
\Bel(x;L)\le B(x;L).
$$
\end{lemma}
\begin{proof}
The proof is standard. First, we fix a test function $w$ on $\R$ and
a dyadic interval $J.$ This gives us a Bellman point $(x;L).$ Then we
start splitting the interval $J,$ while repeatedly applying the main inequality:
\begin{align*}
|J|B(x;L)&\ge|J_+|B(x^{J_+};L^{J_+})+|J_-|B(x^{J_-};L^{J_-})
\\
&\ge\sum_{I\in\D_n}|I|B(x^I,L^I)
=\int_JB\big(x^{(n)}(s);L^{(n)}(s)\big)\,ds\,,
\end{align*}
where $(x^{(n)}(s);L^{(n)}(s))=(x^I;L^I),$ when $s\in I,$ $I\in\D_n.$
By the Lebesgue differentiation theorem, we have
$x^{(n)}(s)\to\big(w(s),w^2(s)\big)$ and by the definition of the
maximal function $L^{(n)}(s)\to(Mw)(s)$ almost everywhere. For bounded $w$
we can pass to the limit and obtain
$$
B(x;L)\ge\av{(Mw)^2(s)}J.
$$
Then, approximating, as before, an arbitrary test function $w$ by its
bounded cut-offs, we get the same inequality for all $w$, which immediately gives
the required property: $B(x;L)\ge\Bel(x;L).$
\end{proof}

\begin{cor}
\label{aboveMF}
For the function $B$ given by~\eqref{BFMF}, the inequality
$$
B(x;L)\ge\Bel(x;L)
$$
holds.
\end{cor}

To prove the converse inequality, we need to construct an optimizer.
However, in the present setting we have no test function realizing the supremum
in the definition of the Bellman function. Thus, an optimizer will be given
by a sequence of test functions.

\begin{figure}[ht]
\begin{center}
\begin{picture}(200,200)
\thinlines
\put(10,0){\vector(0,1){180}}
\put(0,10){\vector(1,0){180}}
\put(60,60){\vector(1,-1){15}}
\linethickness{.5pt}
\thicklines
\qbezier[1000](10,10)(90,10)(170,70)
\qbezier[100](10,10)(90,10)(170,181)
\qbezier[500](150,56)(150,270)(150,56)
\qbezier[1000](81,10)(250,330)(81,10)
\multiput(150,10)(0,3){16}{\circle*{1}}
\multiput(88,10)(0,3){5}{\circle*{1}}
\multiput(10,24)(3,0){27}{\circle*{1}}
\multiput(10,140)(3,0){47}{\circle*{1}}
\put(185,5){\footnotesize $x_1$}
\put(5,185){\footnotesize $x_2$}
\put(148,0){\footnotesize $L$}
\put(0,22){\footnotesize $u^2$}
\put(2,138){\footnotesize $v$}
\put(90,0){\footnotesize $u$}
\put(70,0){\footnotesize $\frac12L$}
\put(40,70){\footnotesize $x_2=vL^{-2}x_1^2$}
\put(88,24){\circle*{2}}
\put(150,140){\circle*{2}}
\put(175,75){\footnotesize $x_2=x_1^2$}
\put(170,155){\footnotesize $x_2=vL^{-1}(2x_1-L)$}
\end{picture}
\caption{
{The extremal trajectory passing through $(u,u^2)$
and $(L,v)$ is tangent to the parabola $x_2=vL^{-2}x_1^2$}
}
\end{center}
\end{figure}
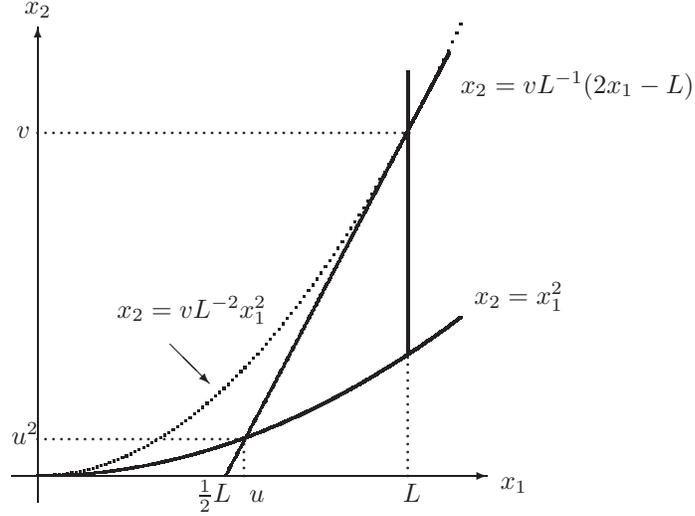

First, we construct an optimizer on $(0,1)$ for the point $(L,v).$
The extremal line passing through this point is $x_2=\frac vL(2x_1-L).$
It intersects the parabolic boundary at the point $(u,u^2)$ with
$$
u=\frac{v-\sqrt{v^2-L^2v}}L=\frac{L\sqrt v}{\sqrt v+\sqrt{v-L^2}}\,.
$$

We need to split the interval $(0,1)$ in half, which splits
the Bellman point $x=(L,v),$ into a pair of points $x^\pm,$
$x=(x^-+x^+)/2.$ We use the homogeneity of the problem in our
construction. We know that the set of test functions for the point
$\tilde x=(\tau x_1,\tau^2x_2)$ is the same as the set of test
functions for the point $x,$ each multiplied by $\tau.$ Therefore,
if $w$ is an optimizer for the point $x,$ then $\tau w$ is an
optimizer for $\tilde x.$ Hence, for the first splitting of $x=(L,v),$
we take the right point $x^+$ not on the continuation of the extremal
line, but on the parabola $x_2=vL^{-2}x_1^2,$ which is tangent to our
extremal line at the point $x.$ Then on the right half-interval
$(\frac12,1)$ we can set the optimizer to be proportional to the
appropriately scaled copy of itself: $w(t)=\beta w(2t-1)$ for
$t\in(\frac12,1).$ What function do we need to take on the left
half-interval? We can split the corresponding Bellman point along
the extremal line in such a manner that the right point
$x^{(\frac14,\frac12)}$ returns to the initial point $x=(L,v)$ and,
therefore, $w(t)=w(4t-1)$ for $t\in(\frac14,\frac12).$ Continuing
in this fashion, we put $w(t)=w(8t-1)$ for $t\in(\frac18,\frac14)$
and so on. We can assume that the first splitting was chosen in such
a way that after $n$ steps the left point $x^-$ lands precisely on
the boundary $x_2=x_1^2,$ and, therefore, on the last interval
$(0,2^{-n})$ we have to set the optimizer $w$ to be constant. Finally,
our optimizing sequence will be given by
\eq[recur]{
w_n(t)=\begin{cases}
\alpha_nL&0<t<2^{-n},
\\
w_n(2^kt-1)&2^{-k}<t<2^{-k+1},\ 1<k<n,
\\
\beta_nw_n(2t-1)&\frac12<t<1.
\end{cases}
}

Let us verify that this recurrent relation defines the
sequence $\{w_n\}$ correctly. To this end, let us introduce a sequence
$\{w_{n,m}\}$ by induction:
$$
w_{n,0}(t)=\begin{cases}
\alpha_nL&0<t<2^{-n},
\\
0&2^{-n}<t<1;
\end{cases}
$$
$$
w_{n,m}(t)=\begin{cases}
\alpha_nL&0<t<2^{-n},
\\
w_{n,m-1}(2^kt-1)&2^{-k}<t<2^{-k+1},\ 1<k<n,
\\
\beta_nw_{n,m-1}(2t-1)&\frac12<t<1.
\end{cases}
$$
We see that $w_{n,m}(t)=w_{n,m-1}(t)$ for all $t$ such that
$w_{n,m-1}(t)\ne0,$ and the measure of the set where $w_{n,m-1}(t)=0$
is $(1-2^{-n})^m,$ i.e. it tends to zero as $m\to\infty.$ Therefore,
$w_{n,m}$ stabilizes almost everywhere as a sequence in $m,$ and its
limit $w_n$ satisfies the recurrent relation~\eqref{recur}.

Now, let us calculate the values of the parameters $\alpha_n$ and $\beta_n.$
We choose them to get $(L,v)$ as a Bellman point of $w_n:$
\begin{align*}
L=\av{w_n}{(0,1)}\!\!
&=2^{-n}\alpha_nL+\big(\frac12-2^{-n}\big)L+\frac12\beta_nL,
\\
v=\av{w_n^2}{(0,1)}\!\!
&=2^{-n}\alpha_n^2L^2+\big(\frac12-2^{-n}\big)v+\frac12\beta_n^2v.
\end{align*}
Solving this system yields
$$
\alpha_n=\sqrt{1+2^{-n+1}}\frac{\sqrt{1+2^{-n+1}}-\sqrt{1-\frac{L^2}v}}
{\frac{L^2}v+2^{-n+1}}\xrightarrow[n\to\infty]{}
\frac v{L^2}\Big(1-\sqrt{1-\frac{L^2}v}\Big)=\frac uL\,.
$$
When solving the quadratic equation for $\alpha_n,$ we chose the minus sign
specifically to get this limit. Choosing the plus sign would produce, instead
of $u,$ the first coordinate of the second intersection point of the
extremal line with the boundary $x_2=x_1^2.$

Now, we need to calculate the maximal function for $w_n,$ which is a simple matter:
$$
Mw_n=\frac{w_n}{\alpha_n}\,.
$$
It is easy to check by induction in $m$ that
$(Mw_{n,m})(t)=\frac{w_{n,m}(t)}{\alpha_n}$ for all $t$ for which
$w_{n,m}(t)\ne0.$ In the limit we obtain the required relation for
$Mw_n.$

Finally, we have
$$
\av{(Mw_n)^2}{}=\frac{\av{w_n^2}{}}{\alpha_n^2}
=\frac{v}{\alpha_n^2}\longrightarrow\frac{vL^2}{u^2}
=\big(\sqrt v+\sqrt{v-L^2}\big)^2=B(L,v;L)\,.
$$

Thus, we have proved the inequality $\Bel(x;L)\ge B(x;L)$ for $x$ on the
line $x_1=L.$ Now, take an arbitrary $x\in\Omega_L$ with $x_1>L/2.$ Let
the extremal line passing through this point intersect the two boundaries
of $\Omega_L$ at the points $(u,u^2)$ and $(L,v)$ and assume that the point
$x$ splits the segment between these two points in proportion
$\alpha\!:\!(1-\alpha).$ Using the main inequality for $\Bel$, linearity
of $B$ on the extremal line, and the just-proved inequality
$\Bel(L,v;L)\ge B(L,v;L),$ we can write down the following chain of estimates:
\begin{align*}
\Bel(x;L)&\ge\alpha\Bel(L,v;L)+(1-\alpha)\Bel(u,u^2;L)
\\
&=\alpha\Bel(L,v;L)+(1-\alpha)L^2
\\
&\ge\alpha B(L,v;L)+(1-\alpha)L^2
\\
&=\alpha B(L,v;L)+(1-\alpha)B(u,u^2;L)=B(x;L)\,.
\end{align*}

We use the same trick to prove inequality $\Bel(x;L)\ge B(x;L)$ for
$x_1\le L/2,$ except now, instead of the vertical extremal line, we use
a nearby line with a large slope. Take a number $\xi$ close to $x_1,$
$\xi<x_1,$ and take the line passing through $x$ and $(\xi,0).$ Let
$(u,u^2)$ and $(L,v)$ be the points where this line intersects the two
boundaries of $\Omega_L.$ Then
$$
v=\frac{L-\xi}{x_1-\xi}\xrightarrow[\xi\to x_1]{}\infty\,,
\qquad\qquad
u=\frac{2\xi\sqrt{x_2}}{\sqrt{x_2}+\sqrt{x_2-4\xi(x_1-\xi)}}
\xrightarrow[\xi\to x_1]{}x_1\,.
$$
The concavity of $\Bel$ implies that
$$
\Bel(x)\ge\frac{x_1-u}{L-u}\Bel(L,v)+\frac{L-x_1}{L-u}\Bel(u,u^2)
\ge\frac{x_1-u}{L-u}\big(\sqrt v+\sqrt{v-L^2}\big)^2+\frac{L-x_1}{L-u}L^2.
$$
The limit of the second term in the last expression is $L^2.$ To calculate
the limit of the first term is a bit of work. First of all, note that
$$
\big(\sqrt v+\sqrt{v-L^2}\,\big)^2
=v\Big(1+\sqrt{1-\frac{L^2}v}\;\Big)^2,
$$
i.e. that term can be rewritten in the form
$$
\frac{x_1-u}{L-u}\,v\,\Big(1+\sqrt{1-\frac{L^2}v}\;\Big)^2
=\frac{x_1-u}{x_1-\xi}\,\cdot\,\frac{L-\xi}{L-u}\,x_2\,
\Big(1+\sqrt{1-\frac{L^2}v}\;\Big)^2.
$$
The limit of the expression in parentheses is $2,$ the second fraction
tends to $1,$ and for the first fraction we have
$$
\frac{x_1-u}{x_1-\xi}=1-\frac{u-\xi}{x_1-\xi}=1-\frac{u^2}{x_2}
\xrightarrow[\xi\to x_1]{}1-\frac{x_1^2}{x_2}\,.
$$
In the end, we have
$$
\Bel(x;L)\ge\big(1-\frac{x_1^2}{x_2}\big)\cdot x_2\cdot4+L^2
=4x_2-4x_1^2+L^2=B(x;L)\,.
$$
Thus, we have proved the following lemma.
\begin{lemma}
$$
\Bel(x;L)\ge B(x;L).
$$
\end{lemma}

Taken together, this lemma and Corollary~\ref{aboveMF} prove the following
theorem:
\begin{theorem}
$$
\Bel(x;L)=
\begin{cases}
\qquad4(x_2-x_1^2)+L^2,&0<x_1\le\frac L2,\ x_2\ge x_1^2,
\\
\big(\sqrt{x_2}+\sqrt{x_2-L(2x_1-L)}\,\big)^2,
&\frac L2\le x_1\le L,\ x_2\ge x_1^2.\rule{0pt}{18pt}
\end{cases}
$$
\end{theorem}
\bigskip

The Bellman setup of the problem discussed in this section was first stated in~\cite{NaTr},
the Bellman function above was found in~\cite{Me} without solving the Monge--Amp\`ere equation
and without using the Bellman function method at all. The consideration presented here appears
in~\cite{SlStVa} (an initial version in~\cite{SlSt}).
\bigskip

{\bf Homework assignments:}
{\small\sl
\begin{enumerate}
\item
{\color{blue} Show that this theorem implies that
$$
\|Mw\|_{L^2(\R)}\le 2\|w\|_{L^2(\R)}.
$$
}
\item
{\color{blue} Follow the same steps to find the Bellman function for the dyadic
maximal operator on $L^p,$ for $p>1.$}
\end{enumerate}
}

\end{document}